\documentclass[11pt,reqno]{amsart}
\usepackage{mathrsfs}
\usepackage{mathrsfs}
\usepackage{amsmath}
\usepackage{amsthm}
\usepackage{amssymb, graphicx, epsfig}
\usepackage{times}
\usepackage{hyperref}
\usepackage[numbers]{natbib}
\usepackage{latexsym}
\usepackage{url}

\numberwithin{equation}{section}

\newcommand\dx{d^+}

\newcommand\arxiv[1]{\url{arxiv:#1}}

\hypersetup{
    bookmarks=true,         
    unicode=false,          
    pdftoolbar=true,        
    pdfmenubar=true,        
    pdffitwindow=true,      
    pdftitle={My title},    
    pdfauthor={Author},     
    pdfsubject={Subject},   
    pdfnewwindow=true,      
    pdfkeywords={keywords}, 
    colorlinks=true,       
    linkcolor=blue,          
    citecolor=blue,        
    filecolor=blue,      
    urlcolor=blue           
}

\newcommand\degree{d}
\newcommand\dout{\degree^+}
\newcommand\xb{\mathbf{x}}
\newcommand\ee{\mathbf{e}}

\newtheorem{thm}{Theorem}[section]
\newtheorem{prop}[thm]{Proposition}
\theoremstyle{definition}
\newtheorem{rem}[thm]{Remark}
\newtheorem{example}[thm]{Example}

\theoremstyle{remark}

\newenvironment{romenumerate}[1][0pt]{
\addtolength{\leftmargini}{#1}\begin{enumerate}
 \renewcommand{\labelenumi}{\textup{(\roman{enumi})}}%
 \renewcommand{\theenumi}{\textup{(\roman{enumi})}}%
 }{\end{enumerate}}

\newcounter{thmenumerate}
\newenvironment{thmenumerate}
{\setcounter{thmenumerate}{0}%
 \def\item{\par
 \refstepcounter{thmenumerate}\textup{(\roman{thmenumerate})\enspace}}
}
{}

\newcommand\gam{\gamma}

\newcommand\gk{\kappa}
\newcommand\gl{\lambda}
\newcommand\gL{\Lambda}

\newcommand\gs{\sigma}
\newcommand\gS{\Sigma}
\newcommand\gss{\sigma^2}

\newcommand{\refT}[1]{Theorem~\ref{#1}}

\newcommand{\refR}[1]{Remark~\ref{#1}}
\newcommand{\refS}[1]{Section~\ref{#1}}
\newcommand{\refE}[1]{Example~\ref{#1}}

\newcommand\punkt{.\spacefactor=1000} 
   
\newcommand\ie{i.e\punkt}
\newcommand\eg{e.g\punkt}

\newcommand\cf{cf\punkt}

\newcommand\qq{^{1/2}}
\newcommand\qqw{^{-1/2}}

\newcommand\set[1]{\ensuremath{\{#1\}}}

\newcommand\xpar[1]{(#1)}
\newcommand\bigpar[1]{\bigl(#1\bigr)}
\newcommand\Bigpar[1]{\Bigl(#1\Bigr)}

\newcommand\lrpar[1]{\left(#1\right)}

\newcommand\E{\operatorname{\mathbb E{}}}
\renewcommand\P{\operatorname{\mathbb P{}}}
\newcommand\Var{\operatorname{Var}}
\newcommand\Cov{\operatorname{Cov}}

\newcommand{\tend}{\longrightarrow}
\newcommand\dto{\overset{\mathrm{d}}{\tend}}
\newcommand\pto{\overset{\mathrm{p}}{\tend}}
\newcommand\asto{\overset{\mathrm{a.s.}}{\tend}}

\newcommand\xfrac[2]{#1/#2}

\newcommand\bbR{\mathbb R}

\newcommand\bbZ{\mathbb Z}

\newcommand\N{\mathcal N}

\newcommand\bst{binary search tree}
\newcommand\mst{\mary{} search tree}
\newcommand\mary{$m$-ary}

\newcounter{CC}
\newcounter{cc}

\newcommand\bX{\mathbf{{Y}}}
\newcommand\bZ{\mathbf{{Z}}}
\newcommand\bmu{\boldsymbol{\mu}}

\newcommand\hbmu{\boldsymbol{\hat\mu}}

\newcommand\cN{\N}

\newcommand\Polya{P{\'o}lya}
\newcommand{\tr}{\operatorname{tr}}

\newcommand\iii{^{(i)}}
\renewcommand\Re{\operatorname{Re}}

\newcommand\bbzp{\mathbb{Z}_{\ge0}}

\newcommand\pfitemx[1]{\par#1:}
\newcommand\pfitemref[1]{\pfitemx{\ref{#1}}}
\newcommand{\Fc}{\cS}
\newcommand\cT{{\mathcal T}}
\newcommand\cS{\mathcal S}
\newcommand\cX{\mathcal X}
\newcommand\cSS{\cS^*}
\newcommand\cSx{\cS^*}
\newcommand\dead{*}
\newcommand\gammm{{\max(\gam_m,0)}}
\newcommand\hD{\hat D}
\newcommand\gkn{(\gk+1)(\gk+2)(\gk+3)} 
\newcommand\aii{A$1'$}

\begin{document}
\title[Normal limit laws for fringe subtrees]
{Multivariate normal limit laws for the numbers of fringe subtrees in
  $ m $-ary search trees and preferential attachment trees} 
\author{Cecilia Holmgren}
\address{Department of Mathematics, Uppsala University, PO Box 480,
SE-751~06 Uppsala, Sweden}
\email{cecilia.holmgren@math.uu.se}
\urladdr{http://katalog.uu.se/empinfo/?id=N5-824}

\author{Svante Janson}
\email{svante.janson@math.uu.se}
\urladdr{http://www.math.uu.se/svante-janson}

\author{Matas \v{S}ileikis}
\address{Department of Applied Mathematics of Faculty of Mathematics and Physics, Charles University in Prague, Malostransk\'{e} n\'{a}m. 25, 118 00 Praha, Czech Republic}
\email{matas.sileikis@gmail.com}


\thanks{CH partly supported by the Swedish Research Council.
SJ partly supported by the Knut and Alice Wallenberg Foundation}

\date{March 26, 2016}
\begingroup
  \count255=\time
  \divide\count255 by 60
  \count1=\count255
  \multiply\count255 by -60
  \advance\count255 by \time
  \ifnum \count255 < 10 \xdef\klockan{\the\count1.0\the\count255}
  \else\xdef\klockan{\the\count1.\the\count255}\fi
\endgroup

\begin{abstract} We study fringe subtrees of random $ m $-ary search
trees and of  preferential attachment trees, by putting them in the context
of generalised \Polya{} urns. In particular we show that for the random $ m
$-ary search trees with $ m\leq 26 $ and for the linear preferential
attachment trees, the number of fringe subtrees that are isomorphic to an
arbitrary fixed tree $ T $ converges to a normal distribution;  more
generally, we also prove multivariate normal distribution results for random
vectors of such numbers for different fringe subtrees. Furthermore, we show
that the number of protected nodes in random $m$-ary search trees for $
m\leq 26 $ has asymptotically a normal distribution.    
\end{abstract}

\maketitle

\section{Introduction}

The main focus of this paper is to consider fringe subtrees of  random 
\mst{s} and of general preferential attachment trees (including the
random recursive tree);
these random trees are defined in \refS{Strees}.
Recall that a \emph{fringe subtree} is a subtree consisting of some node and
all its
descendants, see \citet{Aldous-fringe} for a general theory,
and note that fringe subtrees typically are
``small'' compared to the whole tree. 
(All subtrees considered in the present
paper are of this type, and we will use `subtree' and `fringe subtree' as
synonyms.) 

We will use (generalised) \Polya{} urns to analyze vectors of the numbers
of fringe subtrees of different types
in random $ m $-ary search trees and general (linear) preferential
attachment trees, and in the former class we will also analyze the number of
protected nodes (that is, nodes with distance to a nearest leaf at least two). As a result, we prove multivariate normal asymptotic
distributions for these
random variables,
for \mst{s} when $m\le26$ and for preferential attachment trees with linear
weights. 

\Polya{} urns have earlier been used to
study the total number of nodes in random $ m $-ary search trees, see
\cite{Mahmoud2,Janson,Mahmoud:Polya}. In that case one only needs to
consider an urn with 
$ m-1 $ different types, describing the nodes holding $ i $ keys, where $
i\in \{0,1,\dots, m-2\} $.
For this case it is well-known that asymptotic normality does not hold for
\mst{s} with $m>26$, see \cite{ChernHwang}.
Recently, in \cite{HolmgrenJanson2} more advanced \Polya{} urns 
(with $\binom{2m}{m - 1}$ types) were used to
describe protected nodes in random $ m $-ary search trees.
Only the cases $m=2,3$ were treated in detail in \cite{HolmgrenJanson2}, 
and the cases $m=4,5,6$ were further treated in \cite{Heimburger}.
In \cite{HolmgrenJanson2} a simpler
urn (similar to the urn describing the total number of nodes) was also used
to describe the total number of leaves in random $ m $-ary search trees. 

In
this work we further extend the approach used in \cite{HolmgrenJanson2} for
analyzing arbitrary fringe subtrees of a fixed size in random $ m $-ary
search trees as well as in preferential attachment trees. For the random $ m
$-ary search trees we furthermore extend the methods used in
\cite{HolmgrenJanson2} and in \cite{Heimburger}  to analyze the number of
protected nodes in $m$-ary search trees for $ m\leq 26 $.  

\begin{rem}
The \Polya{} urns yield asymptotic results for the numbers of fringe
subtrees in \mst{s}
for $m\ge27$ too,
using \cite[Theorem 3.24]{Janson} or \cite{Pouyanne2008},
but the normalization is different and the (subsequence) limits are
presumably not normal.
In fact,  our proofs show that for any $m$,
the second largest (in real part) eigenvalue in any of our \Polya{} urns is the
same as the one in the simple \Polya{} urn mentioned above
for counting the total number of nodes,
see \refT{thm:eigenvalues}. As is well-known (see \eg{} \refT{thm:normality}
and \cite{Janson}),
the asymptotic behaviour depends crucially on the second largest eigenvalue,
so we expect the same type of asymptotic behavior
as for the number of nodes, which for $m\geq 27$  is not normal
\cite{ChernHwang}, see also \cite{ChauvinPouyanne}.
We will not consider this case further in the present paper.
\end{rem}

 \subsection{Composition of the paper}
The \mst{s} and general preferential attachment trees are defined in
\refS{Strees}.
Our main results are presented in Section \ref{main}; the results in Section
\ref{mainfringemary} and in Section \ref{protectedmary} concern the case of
random $ m $-ary search trees and the results in Section
\ref{mainfringepref} concern the case of linear preferential
attachment trees.

The results in the case of the random $ m $-ary search
trees  are extensions  of results that
previously have been shown for the special case of the random binary search
tree with the use of other methods, see e.g.,
\cite{Devroye1,Devroye2,HolmgrenJanson}. Furthermore, the results for the $
m $-ary search trees in Section \ref{protectedmary} (where we consider
applications to protected nodes in such trees) are extensions of the results
that were proved for the random binary search trees using other methods in
\cite{MahmoudWard,HolmgrenJanson},  and extensions of both the results and
the methods that were used for $ m=2,3 $ in \cite{HolmgrenJanson2} and for $
m=4,5,6 $ in \cite{Heimburger}.
  The results for the preferential attachment trees in Section
  \ref{mainfringepref} are extensions of results that previously have been
  shown for the random recursive trees, see e.g., \cite{HolmgrenJanson}.
 
  In
  particular we show that for the random $ m $-ary search trees with $ m\leq
  26 $ and for the linear preferential attachment trees, the number of
  fringe subtrees that are isomorphic to an arbitrary fixed tree $ T $
has an asymptotic
  normal distribution;  more generally, we also prove
  multivariate normal distribution results for random vectors of such
  numbers for different trees. 

In Section \ref{generalpolya} we describe
  the theory of generalised \Polya {} urns developed in \cite{Janson} that
  we use in our proofs. 

In Section \ref{polyafringe} we describe the specific \Polya{} urns that we
use for analyzing fringe subtrees in random $ m $-ary search trees, and in
Section \ref{proofs} we use them to  prove the main results for $m$-ary
search trees in Section \ref{mainfringemary}. 
Similarly, in
Section \ref{prefpolya} we describe the specific \Polya{} urns that we use
for analyzing fringe subtrees in preferential attachment trees, and in
Section \ref{proofspref} we use them to prove the main results for
preferential attachment trees in Section \ref{mainfringepref}.
In Section
\ref{protectedpolya} we describe the specific \Polya{} urns that we use for
analyzing protected nodes in random $ m $-ary search trees,
and
in Section
\ref{protectedmaryproof} we use them to prove the result
on protected nodes in  $ m $-ary search trees  in
Section \ref{protectedmary}.

In Section \ref{examples} we present some examples with explicit
calculations.

Finally,
in Section \ref{degree} we use related but simpler \Polya{} urns
to analyze the out-degrees of the nodes in the random trees.

\section{The random trees}\label{Strees}

\subsection{$m$-ary search trees}
We recall the definition of $ m $-ary search trees, see
e.g.\ \cite{Mahmoud:Evolution} or \cite{Drmota}.
An  $ m $-ary search tree, for an integer $ m\geq 2 $, is constructed
recursively
from a sequence of $n$ \emph{keys} (real numbers);
we assume that the keys are distinct.
Each node may contain up to $ m-1 $ keys. 
We start with a tree containing just an empty root.
The first $ m-1 $ keys are put in the root, and are placed in increasing
order from left to right; they divide the set of real numbers into $m$
intervals $J_1,\dots,J_m$.
When the root is full (after the first $ m-1 $ keys are added), it gets $ m$
children that are initially empty,
and each further key is passed to one of the children
depending on which interval it belongs to; a key in $J_i$ is passed to the
$i$'th child.
(The binary search tree, i.e., the case $m=2$,
is the
simplest case where
keys  are passed to the left or right child
depending on whether they are larger or smaller than the key in the root.)
The procedure repeats recursively in the subtrees until all keys are added
to the tree. 

We are primarily interested in the random case
when the keys form a
uniformly random permutation of $\set{1,\dots,n}$,
and we let $ \mathcal{T}_{n} $  denote the random $ m $-ary search tree
constructed from such keys.
Only the order of the keys matter, so alternatively, we may assume that the
keys are $n$ i.i.d.\ uniform random numbers in $[0,1]$. Moreover, considering an infinite sequence of i.i.d. keys, and defining $\mathcal T_n$, for $n = 1, 2, \dots$ as the tree constructed from the $n$ first keys, we obtain a Markov process $(\mathcal{T}_n)_{n = 1}^\infty$.

Nodes that contain at least one key are called \emph{internal}, while empty
nodes  are called \emph{external}. We regard the $m$-ary search tree as
consisting only of the internal nodes; the external nodes are places for
potential additions, and are useful when discussing the tree (e.g.\ below),
but are not
really part of the tree.
(However, the positions of the external nodes are significant. For
example, when a node in a \bst{} has exactly one internal child, we want to
know whether that is a left or a right child.)
Thus, a \emph{leaf} is an internal node that has no
internal children, but it may have external children. (It will have external
children if it is full, but not otherwise.)

From now on, when considering an \mst, we will ignore the values of the keys, but we will keep track of the number of keys in each node.
Hence, a non-random \mst{} is a (finite) ordered rooted tree where each node
is marked with the number of keys it contains, with this number being in
\set{0,\dots,m-1}, and
such that if we include the external nodes, the nodes with $m-1$ keys have exactly $m$ children while the remaining
nodes have no children.

We say that a node (external or internal) with $i\le m-2$ keys has $i+1$ \emph{gaps}, while a full
node has no gaps. It is easily seen that an $m$-ary search tree with $n$ keys
has $n+1$ gaps; the gaps correspond to the intervals of real numbers between
the keys (and $\pm\infty$).

If we condition on the isomorphism class of $\mathcal T_n$ (or even on the underlying permutation), then a new key has the same probability $1/(n+1)$ of being inserted into any of the $n+1$ gaps. Thus the $\mathcal T_{n+1}$ is obtained from $\mathcal T_n$ by choosing gap uniformly at random and inserting a key there.

\begin{rem}\label{Runordered}
In applications where the order of the children of a node does not matter,
we can simplify things by ignoring the order and regard the \mst{} as an unordered
tree. (In this case, we can also ignore the external nodes completely.)

If we treat $(\mathcal T_n)_{n = 1}^\infty$ as a sequence of unordered trees
without external nodes, then  without external nodes $\mathcal T_{n + 1}$ is
obtained from $\mathcal T_{n}$ by choosing a node with probability
proportional to $k + 1 - l$, where $k$ is the number of keys in the node and
$l$ is the number of children of the node (of course $l = 0$, if $k < m-1$),
and giving this node a new key if $k<m-1$ and a new child if $k=m-1$.
\end{rem}

\begin{rem}\label{R00}
Each permutation of \set{1,\dots,n} defines an \mst; however,
different permutations may define the same \mst.
It is possible to obtain a bijection by giving each key in the \mst{} a time
stamp, which is its number in the sequence of keys used to construct the
tree.
(For binary trees we
thus obtain so-called increasing trees, see e.g. \cite{Drmota}.)
This gives a labelled version of \mst{s}. In this context two trees $T$ and
$T'$ are isomorphic if there is an isomorphism with the additional property
that it maps the $i$'th largest time stamp of $T$ to the $i$'th largest
time stamp of $T'$.
\end{rem}

\subsection{General preferential attachment trees}\label{recursive}
Suppose that we are given a sequence of non-negative weights
$(w_k)_{k=0}^\infty$,
with $w_0>0$. Grow a random unordered tree $\Lambda_n$ (with $n$ nodes) recursively, starting
with a single node and adding nodes one by one. Each new node is added as a
child of some randomly chosen existing node;
when a new node is added to $\Lambda_{n-1}$, the
probability of choosing a node $v\in \Lambda_{n-1}$ as the parent is proportional to
$w_{\dx(v)}$, where $\dx(v)$ is the out-degree of $v$ in $\Lambda_{n-1}$.
(More formally, this is the conditional probability, given $\Lambda_{n-1}$ and the
previous history.
The sequence $(\Lambda_n)_{n=1}^\infty$ thus constitutes a Markov process.)

We will mainly consider the case of
\emph{linear  preferential attachment trees}, i.e., when
  \begin{equation}\label{wlinear}
w_k=\chi k+\rho, 	
  \end{equation}
for some real parameters $\chi$ and $\rho$,
  with $\rho=w_0>0$;
this includes the most studied cases of preferential attachment trees.
Note that we obtain the same random trees $T_n$ if we multiply all $w_k$ by
some positive constant.
Hence, only the quotient $\chi/\rho$ matters, and
it suffices to consider $\chi\in\set{1,0,-1}$.
In the case $\chi=-1$, so $w_k=\rho-k$, $w_k$ is eventually negative. This
is not allowed; however,
this is harmless if (and only if) $\rho=m$ is an integer; then $w_m=0$
so no node ever gets more than $m$ children and
thus the values $w_k$ for $k>m$ do not matter and can be replaced by 0.
(We exclude the trivial case $\chi=-1$, $\rho=1$, when $w_1=0$ so no node
ever gets more than one child and the tree $\gL_n$ deterministically is a
path with $n$ nodes.)

\begin{example}\label{ERRT}
The \emph{random recursive tree}
is constructed
recursively by adding nodes one by one, with each new node attached as a
child of a uniformly randomly chosen existing node,
see \cite[Section 1.3.1]{Drmota}.
Hence, this is the case $w_k=1$ for all $k$, which is a
special case of a linear preferential attachment tree \eqref{wlinear}
with $\chi=0$ and $\rho=1$.
(Any $\rho>0$ yields the same tree when $\chi=0$.)
\end{example}

\begin{example}\label{EPORT}
The random \emph{plane oriented recursive tree},
 introduced by \citet{Szymanski},
is constructed similarly
  to the random recursive tree, but we now consider the trees
  as ordered; an existing node with $k$ children thus has $k+1$ positions
  in which a new node can be added, and we give all possible positions of the
  new node the same probability.
The probability of choosing a node $v$ as
  the parent is thus proportional to $\dx(v)+1$, so the plane oriented
  recursive tree is the case of a 
linear preferential attachment tree with $w_k=k+1$, i.e., $\chi=\rho=1$.

This model with $w_k=k+1$ is also the preferential attachment
model by \citet{BarabasiA}, which has become popular and has been studied by
many authors, as has the generalization $w_k=k+\rho$
with arbitrary $\rho>0$,
i.e., \eqref{wlinear} with $\chi=1$.
  \end{example}

\begin{example}\label{EBST}
The \bst{}  is the special case with $w_0=2$, $w_1=1$
and $w_k=0$, $k\ge2$ (and, furthermore, each first child randomly assigned
to be left or right);
as said above, we may regard this as the case $\chi=-1$ and $\rho=2$ of
\eqref{wlinear}.
However, \mst{s} with $m\ge3$ are not preferential attachment trees.
\end{example}

\begin{rem}\label{Rpref0}
  It is often natural to consider preferential attachment trees as
  unordered; it is also possible to consider them as ordered, either by
  assigning random orders as in \refE{EPORT} or by ordering
  the children of each node in the order that they are added to the tree.
\end{rem}

For further descriptions of preferential attachment trees, see e.g.,
  \cite[Section 6]{HJbranching}.

\section{Main results}\label{main}

In this section we state the results on fringe subtrees and protected nodes in random \mst{s} as well as fringe trees in preferential attachment trees.

\subsection{Fringe subtrees in random $ m $-ary search trees}\label{mainfringemary}

\begin{rem}\label{R0}
As said in the introduction, \mst{s} can be regarded as either ordered or
unordered trees; it is further possible to consider the labelled version as in
\refR{R00}.
(See also \cite[Remark 1.2]{HolmgrenJanson} for the special case of the
binary search tree.)
The most natural interpretation is perhaps the one as ordered trees, and it
immediately implies the corresponding result for unordered trees in, for
example,
\refT{multivariate}.
However, in some applications it is preferable to regard the fringe
trees as unordered trees, since this gives fewer types to consider in the
\Polya{} urns that we use, see, e.g., Example
\ref{fringektrees}
and Section \ref{protectedpolya}.
The theorems in this section apply to all these interpretations, via the choice of an appropriate notion of isomorphism.
\end{rem}

The following theorem generalises \cite[Theorem 1.22]{HolmgrenJanson}, where
the special case of the binary search tree was analyzed.

Let $ H_m:=\sum_{k=1}^m 1/k $ be the $ m $'th harmonic number.
Here and below we write $T = T'$ whenever two trees $T$ and $T'$ are
isomorphic.

\begin{thm}\label{multivariate}%
Assume that $2\le m\le26$.
Let\/
$T^1,\dots,T^d$ be a fixed sequence of non-isomorphic non-random
$ m $-ary search trees and let\/
$\bX_{n}=\bigpar{X^{T^{1}}_{n},X^{T^{2}}_{n},\dots,X^{T^{d}}_{n}}$,
where $  X^{T^{i}}_{n}$ is the (random) number of fringe subtrees that are
isomorphic to $T^{i}$ in the random $ m $-ary search tree
$ \mathcal{T}_{n} $
 with $n$ keys.
Let $ k_i $ be the number of keys of\/ $ T^{i} $ for $ i\in\{1,\dots, d\} $.
Let 
$$
\bmu_{n}:=\E\bX_n= \lrpar{\E(X^{T^{1}}_{n}),\E(X^{T^{2}}_{n}),\dots,
\E(X^{T^{d}}_{n})}.
$$
Then
\begin{equation}\label{mv1}
n\qqw (\bX_{n}-\bmu_{n}) \dto \N(0, \Sigma ),
\end{equation}
where
$ \Sigma=(\sigma_{ij})_{i,j=1}^d$ is some covariance matrix.
Furthermore, in \eqref{mv1},
the vector $ \bmu_{n} $ can be replaced by the vector
$\hbmu_n:=n\hbmu$ with
\begin{equation}\label{hmu}
 \hbmu :=\Bigl( \frac{\P(\mathcal{T}_{k_1}=T^{1})}{(H_m-1)(k_1+1)(k_1+2)},
\dots,\frac{\P(\mathcal{T}_{k_d}=T^{d})}{(H_m-1)(k_d+1)(k_d+2)}\Bigr).
\end{equation}
Moreover, if the trees $T^1,\dots,T^d$ have at least one internal node each,
then the covariance matrix $\Sigma$ is non-singular.
\end{thm}

\begin{rem}\label{Rmean}
The fact that $ \bmu_{n} $ can be replaced by the vector $  \hbmu_{n}  $
means that
\begin{equation}\label{rectus}
\E\bigpar{X^{T^{i}}_{n}}
=\frac{\P(\mathcal{T}_{k_i}=T^{i})}{(H_m-1)(k_i+1)(k_i+2)}n +o\bigpar{n\qq}.
\end{equation}
A weaker version of \eqref{rectus} with the error term $o(n)$
follows, for any $m\ge2$,
from the branching process analysis of fringe subtrees in
\cite{HJbranching}, see the proof in \refS{proofs}.
Moreover,
the proof also shows
that \eqref{rectus} holds, for any $m\ge2$,
with the error term $O\bigpar{n^{\max(\gam_m,0)}}$, where
$\gam_m$ is the second largest real part of a root of the polynomial
$\phi_m$ in \refT{thm:eigenvalues}; if $m\le26$ then $\gam_m<\frac12$
(yielding \eqref{rectus}) but if
$m\geq 27$ then $\gam_m>\frac12$, as shown by
\cite{MahmoudPittel} and \cite{FillKapur}.

The vector $\hbmu_n$ can also, using \eqref{xy} below,
be calculated
from an eigenvector of the intensity matrix of the \Polya{} urn defined in
\refS{polyafringe}, see \refT{thm:normality}\ref{T0a}.
See also \cite{JansonMean}.

Also the covariance matrix $ \Sigma= (\sigma_{ij})_{i,j=1}^d$ can be
calculated explicitly from the intensity matrix of the \Polya{} urn,
see \refT{thm:normality}\ref{T0b}--\ref{T0c}.
We give one example in Section \ref{examples}.
The results in \cite{JansonMean} also show
\begin{equation}\label{mv0}
\sigma_{ij}
=\lim_{n\to \infty}\frac{1}{n}\Cov\bigpar{X^{T^{i}}_{n},X^{T^{j}}_{n}}.
\end{equation}
More generally, it follows from the results in
\cite{JansonPouyanne} that all moments converge in \eqref{mv1},
see \refR{RSN}.
Similarly, as a consequence of \cite{JansonPouyanne} and \cite{JansonMean},
moment convergence and, in particular, asymptotics of variance and
covariances as in \eqref{mv0} hold in all theorems in this section.
\end{rem}

\begin{rem}\label{Rsing}
  The covariance matrix may be singular if some $T^i$ is the tree consisting
  of a single external node. For example, if $d=m-1$ and $T^i$ is the tree
  consisting of a single node with $i-1$ keys, and thus $i$ gaps,
then, by counting the number of gaps in the tree $\cT_n$,
\begin{equation}\label{sing}
  \sum_{i=1}^d i X^{T^i}_n=n+1,
\end{equation}
so this sum is deterministic, and thus the covariance matrix is singular.
(In particular, if $m=2$, then the number of external nodes is
deterministic, namely $n+1$.)
Moreover, \eqref{sing} shows that the number of external nodes $X^{T^1}$
is an affine function of the numbers $X^{T^j}$, $j\ge2$; thus it is always
possible to reduce to the case when every tree $T^i$ has at least one
internal node and $\gS$ is non-singular.
\end{rem}

The following theorem is an important corollary of Theorem 
\ref{multivariate}. 
It also follows from \citet[Theorem 5.1]{FillKapur}.
The special case of the random binary search tree
was proved by \citet{Devroye1}, and the covariances for $ Y_{n,k} $ 
in that case were given by \citet{DennertGrubel}, 
see also
\cite[Theorem 1.19 and Proposition 1.10]{HolmgrenJanson}.

\begin{thm}\label{variance} 
Assume that $2\le m\le26$.
Let $ k\ge0 $ be an arbitrary fixed integer.
and let $ X_{n,k} $ be the (random) number of fringe subtrees with $k$ keys
in the random $ m $-ary search tree $ \cT_n $
 with $n$ keys.
Then, as $n\to \infty$, 
 \begin{align}\label{main2a}
n^{-1/2}\bigpar{Y_{n,k}-\E Y_{n,k}} 
&\dto \mathcal{N}(0,\sigma^{2}_k), 
 \end{align}
 where
$\sigma^{2}_k$ is some constant with $\sigma^2_k>0$ except when $k=0$ and $m=2$.
We also have
 \begin{align}\label{main2b}
n^{-1/2}\Bigpar{Y_{n,k}-\frac{n}{(H_m-1)(k+1)(k+2)}}
&\dto \mathcal{N}(0,\sigma^{2}_k).
 \end{align}
 \end{thm}

\begin{rem}
 The asymptotic mean $ \frac{n}{(H_m-1)(k+1)(k+2)} $ in \eqref{main2b}
easily follows from \eqref{rectus}, see the proof in \refS{proofs}.
 The constant $\sigma^{2}_k$ can again be calculated explicitly from our
 proof. 
\end{rem}
We give one example of Theorem \ref{variance} in Section \ref{ex1},
where we let $ m=3 $ and $ k=4 $.

\subsection{Protected nodes in random $ m $-ary search trees}\label{protectedmary}

There are many recent studies of so-called protected nodes in various
classes of random trees, see
e.g.\ 
\cite{Bona,CheonShapiro,DevroyeJanson,DuProdinger, MahmoudWard,Mansour, HolmgrenJanson, HolmgrenJanson2, HJbranching}. 
A node is \emph{protected} (more precisely, two-protected) if it is not a
leaf and  
none of its children is a leaf.  

The following result was proved by using P\'olya urns in 
\cite[Theorem  1.1]{HolmgrenJanson2} for $m=3$ and in \cite{Heimburger} for
$m=4$, $5$ and $6$.
\begin{thm}\label{mainprotected}
Let $ Z_{n} $ be the number of protected
  nodes in the random $ m $-ary search tree $\cT_n$ with $n$ keys.
Then, if $ m\leq 26 $, we have
\begin{equation}\label{prot}
n\qqw\bigpar{Z_{n}-\E Z_{n}}\dto
\N\lrpar{0,\sigma^{2}}, 
\end{equation}
where 
$ \sigma^{2} $ is some positive constant.
Furthermore, $\E Z_{n}$ can be replaced by $\mu n$
with
\begin{equation}\label{muprot}
  \mu:=\frac{1}{m(H_m-1)}\sum_{\ell=0}^{m-1}
  \frac{m!}{(m-\ell)!}\cdot
  \frac{(m(m-\ell))!}{(m(m-\ell)+\ell+1)!}.
\end{equation}
\end{thm}

\begin{rem}\label{Rmeanprot}
The fact that $\E Z_n$ can be replaced by $\mu n$ means that 
\begin{equation}\label{muprot0qq}
\E(Z_{n})=\mu n+o\bigpar{n\qq}.  
\end{equation}
As in \refR{Rmean}, a weaker version 
with $o(n)$
follows for any $m\ge2$   
from \cite{HJbranching}, see the proof in \refS{protectedmaryproof}.
Moreover, our proof shows that 
\eqref{muprot0qq} holds,
for any $m\ge2$, 
with the error term $O\bigpar{n^{\max(\gam_m,0)}}$, with $\gam_m$ as in
\refR{Rmean}. 
\end{rem}

 The constant $\sigma^{2}_m$ can  be calculated explicitly from our proof of
 Theorem \ref{mainprotected}.
For examples of Theorem \ref{mainprotected} with explicit calculations of the asymptotic variance $ \sigma^{2} $, we refer the reader to \cite{HolmgrenJanson2} for $ m=3 $ and \cite{Heimburger} for $ m=4 $.

\subsection{Fringe subtrees in preferential attachment trees}\label{mainfringepref}

The following theorem was proved for the random recursive tree in
\cite[Theorem 1.22]{HolmgrenJanson} using Stein's method.
Here we give a generalisation to the linear preferential attachment trees. The result applies to all three versions of the tree as mentioned in Remark \ref{Rpref0} with the notion of isomorphism chosen appropriately.

\begin{thm}\label{recursivemulti}%
 Let 
$\Lambda^1,\dots,\Lambda^d$ be a fixed sequence of non-isomorphic unordered (or ordered) trees and let 
$\bZ_{n}=(X^{\Lambda^{1}}_{n},X^{\Lambda^{2}}_{n},\dots,X^{\Lambda^{d}}_{n}),$ where $  X^{\Lambda^{i}}_{n}$ is the number of fringe subtrees that are isomorphic to $\Lambda^{i}$ in 
the linear preferential attachment tree $ \Lambda_{n} $. Let $ k_i $ be the number of nodes in $\Lambda^{i}$.
Let $$ 
\bmu_{n}:=\E\bZ_n= \lrpar{\E(X^{\Lambda^{1}}_{n}),\E(X^{\Lambda^{2}}_{n}),\dots,
\E(X^{\Lambda^{d}}_{n})}. 
$$ 
Then 
\begin{equation}\label{multipref}
n\qqw (\bZ_{n}-\bmu_{n}) \dto \N(0, \Sigma ),
\end{equation}
where  the vector $ \bmu_{n} $ can be replaced with the vector
$\hbmu_n:=n\hbmu$ with
\begin{equation}\label{hmupref}
 \hbmu :=\Bigl(\frac{
  \P(\Lambda_{k_1}=\Lambda^{1})\cdot\gk}{(k_1+\gk-1)(k_1+\gk)},\dots,
\frac{\P(\Lambda_{k_d}=\Lambda^{d})\cdot\gk}{(k_d+\gk-1)(k_d+\gk)}\Bigr),   
\end{equation}
with
\begin{equation}\label{kappa}
  \gk:=\frac{\rho}{\chi+\rho}=\frac{w_0}{w_1},
\end{equation}
 and $ \Sigma =(\sigma_{ij})_{i,j=1}^d$ is some non-singular covariance matrix .
\end{thm}
 
Note that for the random recursive tree $ \gk=1 $ and for the plane oriented recursive tree $ \gk=\frac{1}{2} $.
\begin{rem}\label{branchjagerspreferential}
The proof shows also that
\begin{equation}\label{rpref}
\E(X^{\Lambda^{i}}_{n})
=\frac{\P(\Lambda_{k_i}=\Lambda^{i})\cdot\gk}{(k_i+\gk-1)(k_i+\gk)}\,n 
+O(1).
\end{equation}
A weaker version of \eqref{rpref} with the error term $o(n)$
follows from the branching process analysis of fringe subtrees, 
see \cite[(5.29) and Example 6.4, in particular (6.24)]{HJbranching}.  

The vector $ \hbmu$, and thus the coefficient of $n$ in \eqref{rpref}, 
can also be calculated from an eigenvector of the intensity matrix in the
proof; similarly, the covariance matrix $ \Sigma=
(\sigma_{ij})_{i,j=1}^d$ can be calculated explicitly from our proof.
\end{rem}

The following theorem is an important corollary of Theorem
\ref{recursivemulti}. The cases of the random recursive tree 
($\gk=1$) 
and binary search tree ($\gk=2$)
were proved in
\cite[Theorems~4 and~5]{Devroye1}
and the case of the plane oriented recursive tree 
($\gk=\frac12$) was proved in \cite[Theorem 1.1]{Fuchs2012}. 

\begin{thm}\label{variance2} Let $ k $ be an arbitrary fixed integer.
Let $ Y_{n,k} $ be the number of subtrees with $ k $ nodes
in the linear preferential attachment tree $ \Lambda_n $. Then, as 
$n\to\infty$, 
\begin{align}\label{prefksubtrees}
n^{-1/2}\bigpar{Y_{n,k}-\E Y_{n,k}}
&\dto \mathcal{N}(0,\sigma^{2}_k), 
 \end{align}
 where
$\sigma^{2}_k$ is some constant with $\sigma^2_k>0$. Furthermore, we also have
\begin{align}\label{rec}
n^{-1/2}\Bigpar{Y_{n,k}-\frac{\gk}{(k+\gk-1)(k+\gk)}\,n}
&\dto \mathcal{N}(0,\sigma^{2}_k), 
 \end{align}
 with $ \kappa $ as in \eqref{kappa}.
 \end{thm}

\begin{rem}
It follows from \eqref{rpref}, see the proof, that 
\begin{equation}\label{obliquuspref}
\E\bigpar{Y_{n,k}}
=\frac{\gk}{(k+\gk-1)(k+\gk)}\,n +O(1).
\end{equation}
 The constant $\sigma^{2}_k$ can again be calculated explicitly from our
 proof.  
\end{rem}
We give one example in Section \ref{ex3}, where we let $ k=3 $.

\section{Generalised P\'olya urns}\label{generalpolya}

A (generalised) \Polya{} urn process is defined as follows,
see e.g.\ \cite{Janson} or \cite{Mahmoud:Polya}.
There are balls of $ q $ types (or colours) $ 1,\dots, q $, and
for each $ n $ a random vector $ \mathcal{X}_n=(X_{n,1},\dots,X_{n,q}) $,
where $ 
X_{n,i} $ is the number of balls of type  
$ i $ in the urn at time $ n $. The urn starts with a given vector 
$\cX_0$. For each type $ i $, there is an activity (or weight) 
$ a_i \in \mathbb{R}_{\ge0} $, and 
a random vector $ \xi_i=(\xi_{i1},\dots, \xi_{iq}) $. The urn evolves
according to a discrete time Markov process. At each time $ n\geq 1 $, 
one ball is drawn at random from the urn, with the probability of any ball
proportional to its activity. Thus, the drawn ball has  type $ i$  
with probability $ \frac{a_iX_{n-1,i}}{\sum_{j}a_jX_{n-1,j}} $. 
If the drawn ball has type $i$, it is replaced
together with $\Delta X_{n,j}\iii$ balls of type $j$, $j=1,\dots,n$,
where the random vector 
$ \Delta \cX_{n}\iii=(\Delta X_{n,1}\iii,\dots, \Delta X_{n,q}\iii) $
has the same distribution as $ \xi_i$ and is
independent of everything else that has happened so far.  
We allow $\Delta X_{n,i}\iii=-1$, which means that the drawn ball is
  \emph{not} replaced.

Usually, the random variables $X_{n,i}$ and $\xi_{ij}$ are integer-valued,
with $X_{n,i}\ge0$, in accordance with the interpretation as numbers of
balls; we assume this unless we explicitly make an exception.
However, see \refR{Rreal} for an extension, which will be used in
\refS{prefpolya}.

The \emph{intensity matrix} of the \Polya{} urn is
the $ q\times q $  matrix
\begin{equation}\label{A}
A:=(a_j\E\xi_{ji})_{i,j=1}^{q}.   
\end{equation}
The intensity matrix $ A $ with its eigenvalues and
eigenvectors is central for proving limit theorems. 

The basic assumptions in \cite{Janson} are the following.
We say that a type $i$ is \emph{dominating}, if 
every other type $j$ can be found 
with positive probability
at some time in an urn started with a
single ball of type $i$. The urn (and its matrix $A$) is \emph{irreducible} if every type is dominating. 
\begin{itemize}
\item[(A1)]
$ \xi_{ij}\geq 0 $ for $ j\neq i $ 
and $\xi_{ii} \geq -1$.
(I.e.,  the drawn ball may be removed, but no other ball.)
\item[(A2)]
$ \E(\xi_{ij}^{2})<\infty $ for all $ i,j\in\{1,\dots,q\} $.
\item [(A3)] 
The largest real eigenvalue $ \lambda_1 $ of $ A $ is positive.
\item [(A4)]
The largest real eigenvalue $ \lambda_1 $ is simple.
\item [(A5)] 
There exists a dominating type $ i $ with $ X_{0,i}>0 $, i.e., we start with
at least one ball of a dominating type.
\item [(A6)]
$ \lambda_1 $ is an eigenvalue of the submatrix of $A$ given by the
  dominating types.
\end{itemize}
We will also use the following simplifying assumption.
\begin{itemize}
  \item [(A7)]
At each time $ n\geq 1 $, there exists a ball of a dominating type.
\end{itemize}

Before stating the results that we use, we need some notation. 
By a vector $ v $ we
mean a column vector, and we write $ v' $ for its transpose (a row vector).  
More generally, we denote the transpose of a matrix $ A $ by $ A' $. By an
eigenvector of $ A $ we mean a right eigenvector;
a left eigenvector is the same as the transpose of an
eigenvector of the matrix $ A' $. If $ u $ and $ v $ are vectors then $ u'v
$ is a scalar while $ uv' $ is a $ q\times q $ matrix of rank 1. We also use the
notation $ u\cdot v $ for $ u'v $. 
We let $ \lambda_1 $ denote the largest real eigenvalue of $A$. (This exists by our assumptions and the Perron--Frobenius theorem.)
Let $ a=(a_1,\dots,a_q) $ denote the
(column) vector of activities, and let $ u_1' $ and $ v_1 $ denote left and
right eigenvectors of $A$ corresponding to the 
largest eigenvalue $ \lambda_1 $, i.e., vectors satisfying   
\begin{align*}
u_1'A=\lambda_1 u_1',&&& Av_1=\lambda_1 v_1.  
\end{align*}
We assume that $ v_1 $ and $ u_1 $ are normalised so
that 
\begin{align}\label{normalised} a\cdot v_1=a'v_1=v_1'a=1,
&&&
u_1\cdot  v_1=u_1'v_1=v_1'u_1=1,  
\end{align} see \cite[equations (2.2)--(2.3)]{Janson}. 
We write $ v_1=(v_{11},\dots,v_{1q}) $.

We define   
$$P_{\lambda_1} =v_1u_1', $$ and $ P_{I}=I_{q}-P_{\lambda_1} $, where $
I_{q} $ is the $ q\times q $ identity matrix. 
(Thus $ P_{\lambda_1} $ is the one-dimensional projection onto the eigenspace
corresponding to $\gl_1$ 
such that $P_{\gl_1}$ commutes with the matrix $ A $, 
see \cite[equation (2.5)]{Janson}; note that $ P_{\gl_1} $ typically is not
orthogonal).  
We define the matrices 
\begin{align}
\label{Bi} B_i &:=\E(\xi_i\xi_i'),
\\\label{B}
B&:=\sum_{i=1}^{q}v_{1i}a_iB_i,\\
\label{Sigma}
\Sigma_I&:=\int_{0}^{\infty}P_{I}e^{sA}Be^{sA^{'}}P_{I}'e^{-\lambda_1s}ds,
\end{align}
where we recall that $ e^{tA}=\sum_{j=0} ^{\infty}\xfrac{t^{j}A^{j}}{j!}$.
It follows from \cite{Janson}, see also \cite{JansonMean}, that when
$\Re\lambda<\lambda_1/2 $,  
the matrix-valued integral $ \Sigma_I $ in \eqref{Sigma} is absolutely
convergent. 

It is proved in \cite{Janson} that, under assumptions (A1)--(A7),
$\mathcal{X}_n$ is 
asymptotically normal if
$\Re\lambda\le\lambda_1/2 $ for each eigenvalue $ \lambda\neq \lambda_1 $;
more precisely, if 
$\Re\lambda<\lambda_1/2 $ for each such $ \lambda$,
then 
$n^{-1/2} (\mathcal{X}_n-n\mu)\dto \N(0,\Sigma)$ for some $\mu=(\mu_1,\dots, \mu_k)$ and
$\Sigma=(\sigma_{i,j})_{i,j=1}^{k}$. 
(If $ \Re\lambda=\lambda_1/2 $ for some eigenvalue $\gl$, 
then $\mathcal{X}_n$ is still asymptotically normal,
however with another normalisation.) 
The asymptotic covariance matrix $\Sigma$ may be calculated in
different ways; 
we refer to \cite[Theorem 3.22]{Janson} for a general formula,
but we will instead use two simpler formulas that
apply under (different) additional assumptions; see further 
\cite[Section 5]{Janson}.

\begin{thm}[{\cite[Theorem 3.22 and Lemmas 5.4 and 5.3(i)]{Janson}}]
\label{thm:normality}
Assume \textup{(A1)--(A7)} and that  we have normalised as in
\eqref{normalised}. 
Also assume that\/ $\Re\lambda<\lambda_1/2 $,
for each eigenvalue $ \lambda\neq \lambda_1 $.
  \begin{romenumerate}[-10pt]
  \item \label{T0a}
Then, as $ n\to\infty $, 
\begin{equation}\label{t0a}
n^{-1/2} (\mathcal{X}_n-n\mu)\dto \N(0,\Sigma),  
\end{equation}
with 
$\mu=\gl_1v_1$ and some
covariance matrix $ \Sigma $.
  \item \label{T0b}
Suppose further that, for some $ c>0$,
\begin{equation}\label{acdot}
a\cdot\E(\xi_i)=c,   
\qquad i=1,\dots,q.
\end{equation}
Then the
covariance matrix $ \Sigma = c\Sigma_I$, with $ \Sigma_I $
as in \eqref{Sigma}. 
  \item \label{T0c}
Suppose that 
\eqref{acdot} holds and that 
the matrix $ A $ is
diagonalisable, and let
$ \{u_i'\}_{i=1}^{q} $ and $\{v_i\}_{i=1}^{q} $ are dual bases of left
and right eigenvectors, respectively, i.e., $ u_i'A=\lambda_iu_i' $,
 $Av_i=\lambda_iv_i $ and $ u_i\cdot v_j= \delta_{ij} $. 
Then, the covariance matrix in \ref{T0a} is given by,
with the matrix $ B $ as in \eqref{B},
\begin{align}\label{simpleSigma}
\Sigma=c\sum_{j,k=2}^{q} 
\frac{u_j'Bu_k}{\lambda_1-\lambda_j-\lambda_k}v_jv_k'.
\end{align}  
  \end{romenumerate}
\vskip-\baselineskip
\qed
\end{thm}

\begin{rem}\label{Racdot}
It is easily seen that \eqref{acdot} implies that
$\gl_1=c$ and $u_1=a$, see e.g.\ \cite[Lemma 5.4]{Janson}.
\end{rem}

\begin{rem}\label{Rmom}
  From \eqref{t0a} follows immediately a weak law of large numbers:
  \begin{equation}\label{lln}
	\cX_n/n\pto \mu.
  \end{equation}
In fact, the corresponding strong law $	\cX_n/n\asto \mu$ holds as well,
see \cite[Theorem 3.21]{Janson}.
It follows that corresponding strong law of large numbers holds for all
theorems in \refS{main}.

Furthermore, 
in all applications in the present paper,  all $\xi_{ij}$ are bounded and
thus each $X_{n,i}\le Cn$ for some deterministic constant;
hence \eqref{lln} implies by dominated convergence
that also the means converge:
  \begin{equation}\label{means}
	\E\cX_n/n\to \mu.
  \end{equation}
(In fact, this holds in general, without assuming that $\xi_{ij}$ are bounded, 
  since it is easy to see that (A2)  implies that $X_{n,i}/n$ are uniformly
integrable, which together with \eqref{lln} yields \eqref{means},
see e.g.\ \cite[Theorem 5.5.4]{Gut}.)

Moreover,
in all applications in the present paper,  $a\cdot\xi_i=c$ for some $c$ and
every $i$ (a stronger version of \eqref{acdot}), and then \eqref{means} can
be improved, with an explicit rate of convergence, see \cite{JansonMean}.
\end{rem}

\begin{rem}\label{Rreal}
It has been noted several times that the \Polya{} urn process
is also well-defined for \emph{real-valued} 
$X_{ni}$ and $\xi_{ij}$, 
see \eg{} \cite[Remark 4.2]{Janson}, \cite[Remark 1.11]{Janson2006} and
\cite{Pouyanne2008}  
(\cf{} also \cite{Jirina} for the related case of branching processes);
the ``number of balls'' $X_{ni}$ may thus be any non-negative real number.
(This can be interpreted as an urn containing a certain amount (mass) of each
type, rather than discrete balls.)
In this version,  Condition (A1) 
is replaced by the more general:
\begin{itemize}
\item[(\aii)]
For each $i$, either $\xi_{ii}\ge0$ (a.s.), 
or there is a real number $d_i>0$ such
that $X_{0,i}$ and every $\xi_{ji}$ is divisible by $d_i$ (a.s.), and
furthermore $\xi_{ii}\ge -d_i$.
Moreover, in both cases, $\xi_{ij}\ge0$ when $i\neq j$.
\end{itemize}
(Note that (A1), with all variables integer-valued,
is the case $d_i=1$ for every $i$.)
\refT{thm:normality} holds for this version too,
see \cite[Remark 4.2]{Janson}.
(The extra assumptions used there are easy to verify when (\aii) holds
together with (A2)--(A7).)
\end{rem}

In the \Polya{} urns used in this paper,
it is immediately verified that (A1), or at least (\aii), holds, and also (A2).
Furthermore,
it is easily seen (from the 
definitions using trees) that every type with positive activity is
dominating.  If we remove rows and 
columns corresponding to the types with activity 0 from $A$, then 
 the removed columns are identically 0, so the set of non-zero eigenvalues
 of $A$ is not changed. The remaining matrix is irreducible, and using
the Perron--Frobenius theorem, it is easy to verify 
all conditions (A3)--(A6), see \cite[Lemma 2.1]{Janson}. Furthermore,
in our urns there will always be a ball of positive activity,
so essential extinction is impossible and (A7) holds.
Hence, \refT{thm:normality} applies.

\section{\Polya{} urns to count fringe subtrees in random $ m $-ary search trees}\label{polyafringe}

In this section we describe the \Polya {} urns that we will use in the
analysis of fringe subtrees to prove Theorem \ref{multivariate} and Theorem
\ref{variance} for \mst{s}.
The definitions apply to all interpretations of the trees (ordered/unordered, labeled/unlabeled), see Remark \ref{R0}.


Let 
$T^1,\dots,T^d$ be a fixed sequence of 
(non-random) $ m $-ary search trees and let
$\bX_{n}=(X^{T^{1}}_{n},X^{T^{2}}_{n},\dots,X^{T^{d}}_{n}),$ where 
$X^{T^{i}}_{n}$ is the number of fringe subtrees 
in $ \mathcal{T}_{n} $
that are isomorphic to
$T^{i}$. 
We may assume
that at least one tree $T^i$ contains at least $m-2$
keys. (Otherwise we simply add one such tree to the sequence.)

We define a partial order on the set of (isomorphism classes of)
non-random $ m $-ary search trees, such that $T \preceq T'$  if $T'$ 
can be obtained from $T$ by adding keys (including the case $T'=T$).
 Of course, the order depends on the definition of isomorphism (ordered, unordered, labelled) one considers. 

Assume that we have a given $ m $-ary search tree $\cT_n$
together with its external nodes. 
Denote the fringe subtree of $\cT_n$ rooted at a node $v$ by $\cT_n(v)$.
We say that a node $v$ is \emph{living} if 
$ \cT_n(v)\preceq T^{i}$  for some $ i\in \{1,\dots,d\} $,
i.e., if $\cT_n(v)$ is isomorphic to some $T^i$ or can be grown to
become one of them  by adding more keys. 
Note that this includes all external nodes and all leaves with at most $ m-2 $  keys (by the assumption
above). 
Furthermore, we let all descendants of a living node be living.
All other nodes are \emph{dead}.

Now erase all edges from dead nodes to their children.
This yields a forest of small trees, 
where each tree either consists of a single dead node or 
is living (meaning that all nodes are living) and can be grown to become one
of the $T^i$. 
We regard these small
trees as the balls in our generalised \Polya{} urn.
Hence, the types in this \Polya{} urn are all (isomorphism types of) 
non-random $ m $-ary search trees $T$ such that 
$ T\preceq T^{i}$  for some $ i\in \{1,\dots,d\} $,
plus
one dead type.
We denote the set of living types by
\begin{equation}\label{cS}
  \cS:=\bigcup_{i=1}^d\set{T:T\preceq T^i},
\end{equation}
and the set of all types by $\cSS:=\cS\cup\set{\dead}$, where $\dead$ is the
dead type. 
(The set $\cS$ is thus a down-set for the partial order
$\preceq$. Conversely, any finite down-set  
occurs as $\cS$,
provided it contains all trees with a single node and thus $\le m-2$ keys;  
we may simply let $T^1,\dots,T^d$ be the trees
in $\cS$.)

When a key is added to the tree $\cT_n$, 
it is added to a leaf with at most $ m-2 $ keys or to an external node,
and thus to one of the living subtrees in the forest just described.
If the root of that subtree still is living after the addition, then that 
subtree becomes a living subtree of a different type; 
if the root becomes dead, then the subtree
is further decomposed into one or several dead nodes and 
several (at least $m$)
living subtrees. In any case, the transformation does not depend on anything
outside the subtree where the key is added.
The random evolution of the forest obtained by decomposing $\cT_n$ is thus
described by a \Polya{} urn with the types $\cSS$, where each type has
activity equal to its number of gaps, and certain transition rules that in
general are random, since the way a subtree is decomposed (or perhaps not
decomposed) typically depends on where the new key is added.

Note that dead balls have activity 0; hence we can ignore them and consider
only the living types (i.e., the types in $\cS$) 
and we will still have a \Polya{} urn.
The number of dead balls can be recovered from the numbers of balls of other types if it is desired, since the total number of keys is non-random and each dead ball contains $m-1$ keys.

Let $X_{n,T}$ be the number of balls of type $T$ in the 
\Polya{} urn, for $T\in\cS$.
The trees $T^i$ that we want to count correspond to different types in the
\Polya{} urn, but they may also appear as subtrees of larger living trees.
Hence, if
$n(T,T')$ denotes the number of fringe subtrees in $T$ that are
isomorphic to $T'$, then $X^{T_i}_n$ is the linear combination
\begin{equation}\label{xy}
  X^{T^i}_n=\sum_{T\in\cS} n(T,T^i) X_{n,T}.
\end{equation}

The strategy to prove Theorem \ref{multivariate} should now be obvious. We
verify that the \Polya{} urn satisfies the conditions of 
\refT{thm:normality} (this is done in \refS{proofs}); 
then that theorem yields asymptotic normality of
the vectors $(X_{n,T})_{T\in \cS}$, and then asymptotic normality of 
$(X_n^{T^1},\dots,X_n^{T^d})$ follows from \eqref{xy}.

\begin{example}[a \Polya{} urn to count fringe subtrees with $  k $ keys]
\label{fringektrees}
As an important example, let us consider the problem of finding the
distribution of the number of fringe subtrees with a given number of keys,
as in \refT{variance}.
In this case, the order of children
in the tree does not matter so it is easier to regard the
trees as unordered.

So, fix $k\ge m-2$ and let
$ T^{i} $, $ i\in \{1,\dots,d\} $, be the sequence of all $ m $-ary search
trees with at most $ k $ keys. 
This is a down-set, so \eqref{cS} simply yields $\cS=\set{T^i:1\le i\le d}$.
We ignore the dead nodes and consider the urn with only the living types
$\cS$.

In the decomposition of an \mst{} constructed above, a node $v$ is living if and
only if the fringe subtree rooted at $v$ has at most $k$ keys. Hence the 
decomposition consists of all maximal fringe subtrees with at most $k$ keys,
plus dead nodes.

The replacement rules in the \Polya{} urn are easy to describe.
A type $T$ with $j$ keys has $j+1$ gaps, and thus has 
activity $j+1$.
Suppose we draw a ball of type $T$ and $T_1,\dots,T_{j+1}$ are the trees
that can be obtained by adding a key
to one of these gaps in $T$. (Some of $T_i$'s may be equal.)
If $j<k$, then each $T_i$ has at most $k$ keys and is itself a type in the urn, so the drawn ball is replaced by one ball of a type chosen uniformly at random among
$T_1,\dots,T_{j+1}$. On the other hand, if $j=k$, then each $T_i$ has $k+1$ keys
and thus has a dead root; the root contains $m-1$ keys, so after removing it we
are left with $m$ subtrees that together contain $k+1-(m-1)\le k$ keys; hence these subtrees are all living and the decomposition stops there. Consequently,
when $j=k$, the drawn ball is replaced by $m$ balls of the types obtained by
choosing one of $T_1,\dots,T_{k+1}$ uniformly at random and then removing its root.

To find the number of fringe subtrees with  $k$ keys, we sum the numbers
$X_{n,T}$ of balls of type $T$ in the urn, for all types $T$ with exactly
$k$ keys. Note that similarly, using \eqref{xy},
we may obtain the number of fringe subtrees with
$\ell$ keys, for any $\ell\le k$, from the same urn.
This enables us to obtain joint convergence in \refT{variance} for several
different $k$, with asymptotic covariances that can be computed from this urn.

Note that for $k=m-2$, the urn described here consists of $m-1$ types,
viz.\ a single node with $i-1$ keys for $i=1,\dots,m-1$.
This urn has earlier been used
in \cite{Mahmoud2,Janson,Mahmoud:Polya} to study the number of nodes, 
and the numbers of nodes with different numbers of keys,
in an \mst. 
\end{example}

In Section \ref{ex1} we give an example with $ m=3 $ and $ k=4 $;
in that case there are  6 different (living) types in the \Polya{} urn.

\begin{rem}\label{Racdotmary}
The types described by the \Polya{} urns above all
have activities equal to the total number of gaps in the type. 
Since the total number of gaps increases by 1 in each step, 
we have $a\cdot\xi_i=1$ for every $i$, deterministically; 
in particular, \eqref{acdot} holds with $c=1$.
Hence, $\gl_1=1$ and $u=a$ by \refR{Racdot}. 
\end{rem}

\begin{rem}
  In the \Polya{} urns above, a type that is a tree $T$ with $g$ gaps ($g-1$
  keys) has activity $g$, and if a ball of that type is drawn, each gap
  is chosen with probability $1/g$ for the addition of a new key.
Each gap in $T$ thus gives a contribution with weight $g/g=1$ 
to the corresponding column in the intensity matrix $A$ in \eqref{A}.
\end{rem}

\section{Proofs of Theorem \ref{multivariate} and Theorem \ref{variance}}\label{proofs}

As said in \refS{generalpolya}, it is easy to see 
(with the help of \cite[Lemma 2.1]{Janson}, for example) 
that the \Polya{} urns
constructed in 
\refS{polyafringe} 
satisfy (A1)--(A7). 
To apply Theorem \ref{thm:normality} 
it remains to show that\/
$\Re\lambda<\lambda_1/2 $ 
for each eigenvalue $ \lambda\neq \lambda_1 $. We will find the eigenvalues
of $ A $ by using induction on the size of $\cS$, 
the set of (living) types.
For definiteness we consider the version with ordered unlabelled trees; the
versions with unordered trees or labelled trees are the same up to minor
differences that are left to the reader.

Note that there is exactly one type that has activity $ j $ for every 
$ j \in\set{1,\dots, m-1}$. 
(These correspond to the nodes holding $ j-1 $ keys.) 
These types are the $m-1$ smallest in the partial order $\preceq$, and they
always belong to the set $\cS$  constructed 
in \refS{polyafringe}.

Let $q:=|\cS|$ be the number of types in $\cS$, and 
choose a numbering $T_1,\dots,T_q$ 
of these $ q $ types that is compatible with the partial
order $\preceq$. 
For $k\le q$, let
\begin{equation}
\Fc_k := \left\{ T_1, \dots, T_k \right\}, 
\end{equation} 
and note that this is a down-set for $\preceq$.
For $k\ge m-1$, we may thus consider the 
\Polya{} urn with the $k$ types in $\Fc_k$ constructed as in
\refS{polyafringe}. Note that this corresponds to decomposing $\cT_n$ into a
forest with all components in $\Fc_k\cup\set*$.
Furthermore, let $ \mathcal{X}_n^{k}:=(X^k_{n,1},\dots,X^k_{n,k}) $, where 
$X^k_{n,i} $ is the number of balls of type  
$T_i$ in the urn with types $\Fc_k$
at time $ n $ and let $ A_k $ be the intensity matrix of this \Polya{} urn.
Thus $A=A_q$.

First let us take a look at the diagonal values $\xi_{ii}$. 
\begin{prop} \label{prop:gaps}
  \begin{thmenumerate}
  \item 
Let $m \ge 3$ and $m-1\le k\le q$.
Then $(A_k)_{ii} = -a_i$
for every type $i = 1, \dots, k$.
Hence the trace satisfies
\begin{equation}\label{trace>2}
  \tr(A_k)=-\sum_{i=1}^k a_{i}.
\end{equation}

  \item 
Let $m =2$ and $1\le k\le q$.
Then $(A_k)_{ii} = -a_i$
for every type $i = 1, \dots, k$, except for the cases
when $T_i$ is the longest left path in $\Fc_k$ 
or the longest right path in $\Fc_k$.
If $k\ge3$ these two exceptional cases are distinct, and 
$(A_k)_{ii} =-a_i+1$ for them;
if $k=1$ or $k=2$, then the exceptional cases coincide and 
$(A_k)_{ii} = -a_i+2$ for the single exceptional case.
Consequently, for any $k\ge1$,
the trace satisfies
\begin{equation}\label{trace=2}
  \tr(A_k)=2-\sum_{i=1}^k a_{i}.
\end{equation}
  \end{thmenumerate}
\end{prop}

\begin{proof}
Observe that
if we draw a ball of type $i$ with $k_i$ keys,
then the ball is replaced either by a single ball of a type with 
$k_i+1$ keys
or by several 
different balls obtained by decomposing a tree with $k_i+1$ keys that has a dead
root.
In the latter case, $m-1$
of the keys are in the dead root, so each living tree in the decomposition has
at most $k_i+1-(m-1)=k_i-m+2$ keys.
Hence, if $m\ge3$, then in no case will there 
be a ball with exactly $k_i$ keys among the added balls, 
and in particular no ball of type $i$; 
consequently, $\xi_{ii}=-1$ and $(A_k)_{ii}=-a_i$,
see \eqref{A}.  

When $m=2$, the same holds except if $T_i$ is such that 
it is possible to add a new key such that the root dies and the tree
decomposes into the dead root, a copy of $T_i$ and an external node.
This can happen only if the root has at most one child, and it follows
by induction that every node has at most one child, so $T$ is a path,
and, furthermore, that
$T_i$ must be either a left path or a right path, with the new key added at
the end; furthermore, it must be the longest such path in $\Fc_k$, since
otherwise the root would not die. If $k\ge3$, then $T_3$ is a path with two
nodes, and it follows that the two exceptional cases are distinct.
If $T_i$ is one of them, then $T_i$ has $a_{ii}$ gaps and only one of them
will yield a new copy of $T_i$ if the new key is added there.
Hence, $\E\xi_{ii}=-(a_{ii}-1)/a_{ii}$, and \eqref{A} yields
$(A_k)_{ii}=-(a_{ii}-1)$. The cases $k=1$ or 2 are similar, but they are so
simple that they are simplest treated separately; the matrices $A_k$ are
$
\begin{pmatrix}
  1
\end{pmatrix}
$
and
$\lrpar{
\begin{smallmatrix}
  -1 & 2 \\
  {\phantom-}1 & 0
\end{smallmatrix}}
$, with $a_1=1$ and $a_2=2$.
\end{proof}

\begin{thm}
\label{thm:eigenvalues}
Let $m \ge 2$. 
The eigenvalues of $A$ are the $m-1$ 
roots of the polynomial 
$\phi_m(\lambda) :=\prod_{i = 1}^{m-1} (\lambda + i) - m!$ plus the multiset 
\begin{equation}\label{ev}
\left\{ -a_i : i = m, m + 1, \dots, q \right\}.
\end{equation} 
\end{thm}
\begin{proof}
We prove by induction on $k$ that 
the theorem holds for $A_k$ (with $q$ replaced by $k$ in
\eqref{ev}), for any $k$ with $m-1\leq k \le q$. The theorem is the case $k=q$.

First, for the initial case $k=m-1$,
$T_i$ is a single node with $i-1$ keys, $i=1,\dots,k$;
thus $X^{m-1}_{n,i}$ is the number of nodes with $i-1$ keys, i.e., the number
of nodes with $i$ gaps. 
(In particular, $X^{m-1}_{n,1}$ is the number of external nodes.)
This \Polya{} urn with $m-1$ types
has earlier been analyzed, see e.g., 
\cite[Example 7.8]{Janson} and
\cite[Section 8.1.3]{Mahmoud:Polya}. 
The $ (m-1)\times (m-1 )$ matrix $ A_{m-1}$ has elements
$ a_{i,i}=-i$ for $ i\in \{1,\dots,m-1\} $, 
 $ a_{i,i-1}=i-1$ for $ i\in \{2,\dots,m\} $, 
$ a_{1,m-1}=m\cdot (m-1) $ and all other elements $ a_{i,j}=0 $,
i.e.,
 \begin{align}\label{AWmatrix}
A_{m-1}=\lrpar{
  \begin{array}{rrrccc}
 -1 & 0 & 0 & \dots & 0 & m(m-1)  \\
1 & -2 & 0 & \dots & 0 & 0 \\
0 & 2 &-3 & \dots & 0 & 0 \\
0 & 0 & 3 & \dots & 0 & 0 \\
\vdots & \vdots & \vdots & \ddots & \vdots & \vdots \\
0 & 0 & 0 &\dots & m-2 & -(m-1) 
\end{array}}
.
\end{align}
As is well-known, 
the matrix $A_{m-1}$ has characteristic polynomial
$\phi_m(\gl)$; this shows the theorem for $k=m-1$, since the set \eqref{ev}
is empty in this case.

We proceed to the induction step.
Let $m-1\le k<q$.
By using  arguments similar  to those that were used in the proof of
\cite[Lemma 5.1]{HolmgrenJanson2} we will show that $A_{k+1}$ inherits (with
multiplicities) the eigenvalues  of $A_k$. 
We write $ a^k=(a_1,\dots,a_k) $ for the activity vector of the \Polya{}
urn with types in $ \Fc_k $.

We have $ \Fc_{k+1}=\Fc_k \cup \set{T_{k+1}}$. 
The vector $\mathcal{X}_n^{k+1}$ determines also the number
of subtrees of each type in the decomposition of $\cT_n$
into the types in $ \Fc_k $, and
there is an obvious linear map
$T:\bbR^{k+1}\to\bbR^{k}$ such that 
$\cX_n^{k} = T \mathcal{X}_n^{k+1}$.
Furthermore, starting the urns with an arbitrary (deterministic) non-zero
vector $\mathcal{X}_0^{k+1}\in\bbzp^{k+1}$ and $\cX_0^{k}=T\mathcal{X}_0^{k+1}$, 
the urn dynamics yield
\begin{align}
  \E(\mathcal{X}_{1}^{k+1}-\mathcal{X}_0^{k+1})
&=\frac{A_{k+1} \mathcal{X}_0^{k+1}}{a^{k+1}\cdot \mathcal{X}_0^{k+1}},
\\
  \E(\mathcal{X}_{1}^{k}-\mathcal{X}_0^{k})&=\frac{A_k \mathcal{X}_0^{k}}{a^{k}\cdot \mathcal{X}_0^{k}}.
\end{align}
Consequently, since also
$a^{k+1}\cdot \mathcal{X}_0^{k+1}=a^k\cdot \mathcal{X}_0^{k}$ (this is the total
activity, i.e., the total number of gaps),
\begin{align*}
\begin{split}
TA_{k+1}\mathcal{X}_0^{k+1} &= (a^{k+1}\cdot \mathcal{X}_0^{k+1})T\E(\mathcal{X}_{1}^{k+1}-\mathcal{X}_0^{k+1})
=(a^k\cdot \mathcal{X}_0^{k})\E(\mathcal{X}_{1}^{k}-\mathcal{X}_0^{k}) = A_k \mathcal{X}_0^{k} 
\\&
= A_kT\mathcal{X}_0^{k+1},
  \end{split}
\end{align*}
and thus, since $\cX_0^{k+1}$ is arbitrary, 
\begin{equation}\label{TA}
TA_{k+1} = A_k T.   
\end{equation}

 Let $u'$ be a left generalised eigenvector of rank $ m $ corresponding to the eigenvalue  $\gl$ of the matrix $A_k$, i.e.,
 $$u'(A_k-\lambda I_k)^{m} =0. $$ 
Then, by \eqref{TA},
 $$u'T(A_{k+1}-\lambda I_{k+1})^{m}= u'(A_{k}-\lambda I_{k})^{m}T=0, $$ 
 and thus 
$u'T=(T'u)'$ is a left generalised eigenvector of $A_{k+1}$
for the eigenvalue $\gl$.
Since $T$ is onto (it maps $(x_1, \dots, x_k,0)$ to $(x_1, \dots, x_k)$), $T'$ is injective and thus
$T'$ is an injective map of the generalised eigenspace (for $\gl$)
of $A_k$ into the generalised eigenspace of $A_{k+1}$.
This shows that $\gl$ is an eigenvalue of 
$A_{k+1}$ with algebraic multiplicity at least as large as for $A_k$. 
Consequently, 
if $A_k$ has
eigenvalues $\lambda_1, \dots, \lambda_k$ (including repetitions, if any),
then
$A_{k+1}$ has
eigenvalues $\lambda_1, \dots, \lambda_k, \lambda_{k+1}$ for some
complex number $\lambda_{k+1}$. 

Then the result follows by the following observation. 
The trace of a matrix is equal to the sum of the eigenvalues; hence, 
\begin{equation}\label{eqtrace}
\tr A_{k+1}=
\lambda_1 + \dots + \lambda_{k+1} =  \tr A_k +\lambda_{k+1}
\end{equation}
and thus by \eqref{trace>2} (when $m>2$) or \eqref{trace=2} (when $m=2$),
\begin{equation}
  \lambda_{k+1}=\tr(A_{k+1})-\tr(A_k)=-a_{k+1}.
\end{equation}
Thus, by induction, Theorem \ref{thm:eigenvalues} holds for every $A_k$,
with $m-1\le k\le q$, and in particular for $A=A_q$.
\end{proof}

\refT{thm:eigenvalues} shows  that the eigenvalues of $A$ are the
roots of $\phi_m$ plus some negative numbers $-a_i$; hence the condition
$\Re\gl<\gl_1/2$ in \refT{thm:normality} is satisfied for all eigenvalues of
$A$ except $\gl_1$ if the condition is satisfied for the roots of $\phi_m$
except $\lambda_1$.
Furthermore, $\gl_1=1$
by \refR{Racdotmary}.
Let 
\begin{equation}\label{gamm}
\gam_m:=\max_{\phi_m(\gl)=0:\gl\neq\gl_1}\Re\gl , 
\end{equation}
i.e., the largest real part of a root of $\phi_m$ except $\gl_1$.
(If $m=2$, when there is no other root, we  interpret this as
$\gam_2=-\infty$.) Thus our condition in \refT{thm:normality} is satisfied
if and only if $\gam_m<\frac12$; 
it is well-known that this holds if $m\le26$, but not
for larger $m$, see \cite{MahmoudPittel} and \cite{FillKapur}.

In the remainder of this section we assume $m\le26$.
Thus 
\begin{equation}\label{gammm}
 \Re\gl\le\gammm<\frac12=\frac{\gl_1}2
\end{equation}
 for every eigenvalue $\gl\neq\gl_1$, and
\refT{thm:normality} applies to the urn defined above.

\begin{proof}[Proof of \refT{multivariate}]
  By \refT{thm:normality}\ref{T0a},
\eqref{t0a} holds, with $\mu=\gl_1v_1=v_1$.

By \eqref{xy}, 
$\bX_{n}=\bigpar{X^{T^{1}}_{n},X^{T^{2}}_{n},\dots,X^{T^{d}}_{n}}=R\cX_n$
for some (explicit) linear operator $R$. Hence, \eqref{t0a} implies
\begin{equation}\label{dixi}
  n\qqw\bigpar{\bX_n-nR\mu}
=R\bigpar{n\qqw(\cX_n-\mu)}
\dto \N\bigpar{0,R\gS R'}.
\end{equation}
Furthermore, as said above, $\Re\gl\le\gam_m$ for every eigenvalue
$\gl\neq\gl_1$. We note also that if $\gl\neq\gl_1$ is an eigenvalue
with $\Re\gl=\gammm$, then $\gl$ is not in \eqref{ev} so $\gl$ is a root of
$\phi_m$; 
furthermore, all roots of $\phi_m$ are simple \cite{MahmoudPittel}, 
and therefore $\gl$ is a simple eigenvalue.
Hence,  
by \cite{JansonMean}, 
\cf{} \cite[Theorem 1]{MahmoudPittel} for a special case proved by
other methods,
\begin{equation}\label{tollm}
  \E\cX_n=
n\mu+O\bigpar{n^\gammm},
\end{equation}
and thus, since $\gam_m<\frac12$ for $m\le26$ as said above,
\begin{equation}\label{tollis}
  \E\cX_n
=
n\mu+o\bigpar{n\qq}.
\end{equation}
Hence,
\begin{equation}\label{vox}
\bmu_n=\E\bX_n=R\bigpar{\E\cX_n}=nR\mu+o\bigpar{n\qq}.
\end{equation}
Consequently, \eqref{dixi} implies \eqref{mv1}, with the covariance matrix $R\gS
R'$,
where $\gS$ is as in \eqref{t0a}.

Moreover, as said in \refR{Rmean}, it follows from
\cite{HJbranching}, to be precise by combining
\cite[(5.30), Theorem 7.10 and Theorem 7.11]{HJbranching},
that (for any $m\ge2$) 
\begin{equation}\label{mensa}
\E\bX_n=n\hbmu+o(n).
\end{equation}
By combining \eqref{vox} and \eqref{mensa} we see that $R\mu=\hbmu$ (since
neither depends on $n$), and thus \eqref{vox} yields \eqref{rectus}.

To see that the covariance matrix $R\gS R'$ is non-singular
when each $T^i$ has an internal node so $k_i>0$, suppose that, on
the contrary,
$u'R\gS R'u=0$ for some vector $u\neq0$. 
Then, by \cite[Theorem 3.6]{JansonMean},
$u'\bX_n=u'R \cX_n$ is deterministic for every $n$.
We argue as for the case $k=2$ in the proof of \cite[Lemma 3.6]{HolmgrenJanson}.
We may assume that every $u_i\neq0$, since we may just ignore each $T^i$
with $u_i=0$; we may also assume that $1\le k_1\le k_2\le\dots$.  
Let $N$ be a large integer, with $N>k_d$, and let $T_1$ be a tree consisting
of a single path with $N+k_1$ internal nodes, each of them (except the root)
the right-most child of the preceding one. 
Let $T_2$ consist of a similar right-most path with $N$ internal nodes,
together with $m-1$ copies of $T_1$, which have their roots as the $m-1$
first children of $T_2$. Note that both $T_1$ and $T_2$ have $(N+k_1)(m-1)$
keys, so they are possible realizations of $\cT_{(N+k_1)(m-1)}$. Moreover,
for any tree $T^i$, $i\ge2$, $T_1$ and $T_2$ have the same number of fringe
trees isomorphic to $T^i$,  while $T_1$ contains $m-1$ more copies of $T^1$
than $T_2$ does. Hence the linear combination $u'\bX_n=\sum_i u_i X_n^{T^i}$
may take at least two different values when $n=(N+k_1)(m-1)$,  
which is a contradiction. Consequently, the covariance matrix cannot be
singular when all $k_i>0$.
\end{proof}
  
\begin{proof}[Proof of \refT{variance}]
Let $T^1,\dots,T^d$ be all non-random \mst{s} with $k$ keys.
(We may consider either ordered or unordered trees; 
for numerical calculations, the unordered case is simpler.)
Then $Y_{n,k}=\sum_{i=1}^d X_n^{T^i}$.  The result \eqref{main2a} thus
follows from \eqref{mv1}, and $\gs_k^2>0$ when $k>0$  
because the covariance matrix $\gS$ in \eqref{mv1} then is non-singular.
Also in the case $k=0$ and $m>2$ we have $\gss_k>0$, because 
$Y_{n,0}=n+1-\sum_{j=1}^{m-2}(j+1)Y_{n,j}$ by \eqref{sing}, and 
the asymptotic variance of the right-hand side is non-zero by an application
of \refT{multivariate} to the sequence of all \mst{s} with at least 1 and at
most $m-2$ keys.

Furthermore, summing \eqref{rectus} over the trees $T^i$  yields
\begin{equation}\label{obliquus}
\E\bigpar{Y_{n,k}}
=\frac{n}{(H_m-1)(k+1)(k+2)} +o\bigpar{n\qq}.
\end{equation}
and thus \eqref{main2b} follows.
\end{proof}

\begin{rem}
  \label{RSN}
The results on moment convergence
in \citet{JansonPouyanne} are stated only for a \Polya{} urn
with deterministic replacements. In the present case, we can instead
consider an urn where the balls represent gaps in the trees in the
construction in \refS{polyafringe}. This yields an urn with
deterministic replacements, although now also some subtractions of
balls. However, the results in
\cite{JansonPouyanne} still apply, 
and thus yield convergence of all moments in
\eqref{mv1}, as said in \refR{Rmean}.
As a consequence, \eqref{main2a} and \eqref{main2b} also hold with
convergence of all moments.
\end{rem}
  
\section{\Polya{} urns to count fringe subtrees in preferential attachment trees} \label{prefpolya} 

We now describe the \Polya{} urns that we use for proving Theorem
\ref{recursivemulti} for linear preferential attachment trees.
We may consider either ordered or unordered trees, 
see Remark \ref{Rpref0}.

\subsection{A \Polya{} urn with infinitely many types for the general case}
\label{infpref}

Let $\gL^1,\dots,\gL^d$ be a fixed sequence of rooted trees and let
$\bZ_{n}=(X^{\Lambda^{1}}_{n},X^{\Lambda^{2}}_{n},\dots,X^{\Lambda^{d}}_{n})$,
where $  X^{\Lambda^{i}}_{n}$ is the number of fringe subtrees that are
isomorphic to $\Lambda^{i}$ in  
$ \Lambda_{n} $. 

Assume that we have a given preferential attachment tree $\Lambda_n$.
As in \refS{polyafringe}, we say that a node $v$ is \emph{living} if 
the fringe subtree
$ \Lambda_n(v)\preceq \Lambda^{i}$  for some $ i\in \{1,\dots,d\} $.
Furthermore, we let all descendants of a living node be living.
All other nodes we declare \emph{dead}.

Now erase all edges from dead nodes to their children.
This yields a forest of small trees, 
where each tree either consists of a single dead node or 
is living and can be grown to become one
of the $\Lambda^i$. 
Again, we regard these small trees as the balls in our generalised \Polya{} urn.
However, unlike the situation in \refS{polyafringe}, we now cannot ignore 
the dead nodes, since they may get new children; furthermore, the
probability of this depends on their degree. Hence we label each dead node
by the number of children it has in $\Lambda_n$.

Hence, the types in this \Polya{} urn are all (isomorphism types of) 
rooted trees $\Lambda$ such that 
$ \Lambda\preceq \Lambda^{i}$  for some $ i\in \{1,\dots,d\} $ (these are
called \emph{normal types}),
plus one type $*_k$ for each positive integer $k$, consisting of a single
dead node labelled by $k$ (these are called \emph{special types}).
In other words, the set of types is $\cS'\cup\cS''$, where
\begin{equation}\label{cS'}
  \cS':=\bigcup_{i=1}^d\set{\Lambda:\Lambda\preceq \Lambda^i}
\end{equation}
is the set of normal types (\cf{} \eqref{cS})
and $\cS'':=\set{*_k:k\ge1}$ is the set of special types.
(As in the corresponding case in \refS{polyafringe},
the set $\cS'$ is thus a down-set for the partial order
$\preceq$. Conversely, any finite down-set  
occurs as $\cS'$, for example, by choosing $\Lambda^1,\dots,\Lambda^d$ to be the trees in $\cS'$.)

Unfortunately, $\cS''$ is infinite, so this is a \Polya{} urn with
infinitely many types. \refT{thm:normality} does not apply to such urns,
and we do not know any extension to infinite-type urns that can be used here.
However, in the linear case \eqref{wlinear}, we can reduce the urn to a
finite-type one; this is done in the following subsection.

Nevertheless, it is easy to describe the replacement rules for this urn in general.
The activity $ a_\gL $ of a normal type $\gL$ is the total weight 
$w_\gL:=\sum_{v\in \gL}w_{\dout(v)}$
of all nodes
in $\gL$, while a special type $*_k$ has activity $w_k$.
If a ball of a normal type $\gL$ is drawn, we add a new child to one of its
nodes, with probabilities determined by the weights $w_k$ and the out-degrees as
in the definition of $\gL_n$, and if the resulting tree is dead (does not
belong to $\cS'$), it is
decomposed into several trees (including at least one dead node). 
If a ball of a special type $*_k$ is drawn,
it is replaced by one ball of type $*_{k+1}$ and one ball of type $\bullet$, the
tree with a single node. 
(The tree $\bullet$ is always a normal type.)

\subsection{A \Polya{} urn with finitely many types for the linear case}
\label{orderpref}

Consider from now on only the linear case \eqref{wlinear}.
We then can replace the infinite-type 
\Polya{} urn in \refS{infpref} by a \Polya{} urn with finitely many types
by using a version of the trick used in \cite{SJ155,Janson} to study
node degrees in random recursive trees and plane oriented recursive trees.

Recall that we may assume that $\chi \in \left\{ -1, 0, 1 \right\}$.
For simplicity, consider first the case when $\rho$ (and thus every $w_k$) is  
an integer. Change each ball of type $*_k$ to $w_k$ balls of a new type $*$.
Let $*$ have activity 1; then the activities are preserved by the
change. Moreover, if a ball of type $*_k$ is drawn, it is, as said
above, replaced by one ball of type $*_{k+1}$ and one of type $\bullet$;
after the change, this means that the number of balls $*$ is increased
by $w_{k+1}-w_k=\chi$.
(This is where the linearity of the weights is essential.) 
If $\chi = -1$, this means that the ball is not replaced.
Consequently, the new urn also evolves as a \Polya{} urn with types
$\cSx:=\cS'\cup\set*$. 
A normal type $\gL$ has the activity $\sum_{v\in \gL}w_{\dout(v)}$ as above in
\refS{infpref}, and when drawn, it is replaced as above, 
except that instead of any ball of a type $*_k$ we add $w_k$ balls of type $*$.
The special type $*$ has activity 1, and when drawn, it is replaced 
and we add
$\chi$ additional balls of type $*$ and one ball of type $\bullet$. 
(For an example, see \refS{ex3}.)

Moreover, we can do the same for general $\rho$.
As said in \refR{Rreal},
the \Polya{} urn process is well-defined for \emph{real-valued} 
$X_{ni}$ and $\xi_{ij}$, interpreting
the ``number of balls'' as the mass in the urn of each type.

With this interpretation, the \Polya{} urn with types $\cSx$ just
described exists and describes the evolution of fringe subtrees for any linear
preferential attachment model, also with non-integer $\rho$.
(We can also allow $\chi\notin\set{-1,0,1}$, but as said earlier, this is
not more general.)

Note that the extension to real-valued urns is needed only when $\chi = 1$;
when $\chi=0$ we may assume that $\rho=1$ and when $\chi = -1$ we necessarily
have $\rho\in\bbZ_+$, see \refS{recursive}; hence the integer
version is enough in these cases. 
Note further that then
(\aii) holds, with $d_i=1$ for every normal type and $\xi_{**}\ge0$.
Hence, \refT{thm:normality} holds also for the real-valued urns
considered here; see \refR{Rreal}.

\begin{rem}
In the linear case, the total activity increases by $\chi+\rho=w_1$ each
time a ball is drawn; the easiest way to see this is by going back to the preferential
attachment tree $\gL_n$ and noting that: the total activity in the urn equals
the total weight of the tree; when we add a new node, it contributes extra 
weight $w_0=\rho$ and the weight of its parent increases by $\chi$.
In other words, we have $a\cdot\xi_i=\chi+\rho$ deterministically, and thus
\eqref{acdot} holds with $c=\chi+\rho$. Consequently, $\gl_1=\chi+\rho$, see
\refR{Racdot}. 

As a consequence, we also see that the total activity of the urn before the
$n$'th step, \ie, the total weight of the preferential attachment tree
$\gL_n$ with $n$ nodes, is
\begin{equation}\label{totalw}
 w_{\gL_n}=
\rho+(n-1)(\chi+\rho)=n(\chi+\rho)-\chi,
\end{equation}
see  e.g., \cite[(6.22)]{HJbranching};
this is deterministic, and any tree with $n$ nodes has the same total weight.
\end{rem}

\section{Proofs of Theorem \ref{recursivemulti} and Theorem \ref{variance2}}\label{proofspref}

To apply Theorem \ref{thm:normality} in the proofs of Theorem \ref{recursivemulti} and Theorem \ref{variance2} it remains to show that\/ $\Re\lambda<\lambda_1/2 $
for each eigenvalue $ \lambda\neq \lambda_1 $ 
of the intensity matrix $A$ for the urn in \refS{orderpref}. 
We will again (as in the case
of random $ m $-ary search trees in \refS{polyafringe}) find the eigenvalues
of $ A $ by using induction. 

We consider either ordered or unordered trees; the results and proofs below
apply to both versions.

Recall that the set of types is $\cSx=\cS\cup\set*$.
Let $q:=|\cSx|$ be the number of types, and 
choose a numbering $T_1,\dots,T_{q-1}$ 
of the $q-1$ normal types that is compatible with the partial
order $\preceq$. (Hence, $T_1=\bullet$.)
For $2\le k\le q$, let
\begin{equation}
\Fc_k := \left\{ T_1, \dots, T_{k-1},* \right\}.
\end{equation} 
Since $\Fc_k\setminus\set*=\set{T_1,\dots,T_{k-1}}$ is a down-set for $\preceq$,
we may thus consider the 
\Polya{} urn with the $k$ types in $\Fc_k$, defined as in \refS{orderpref}.
Let $ A_k $ be the $k\times k$ intensity matrix of this \Polya{} urn, and
let $ \mathcal{X}_n^{k}:=(X^k_{n,1},\dots,X^k_{n,k}) $, where 
$X^k_{n,i} $ is the number of balls of type  
$T_i$ in this urn at time $ n $.
The activities are $a^k:=(a_1,\dots,a_{k-1},1)$, where 
\begin{equation}\label{tw2}
a_i=w_{T_i}=|T_i|(\chi+\rho)-\chi.
\end{equation}
see \eqref{totalw}. (Recall that $*$ always has weight 1.)

\begin{prop} \label{prop:pref}
Let $2\le k\le q$.
  \begin{romenumerate}[-10pt]
  \item \label{Ppref1}
For every normal type $T_i$, $i = 1, \dots, k-1$, except the type that is a
path (rooted at an endpoint) of maximal length,
\begin{equation}\label{pref1}
(A_k)_{ii} = -a_i.
\end{equation}
\item \label{Ppref2}
For the normal type $T_j$ that is a path of maximal length among all paths
in $\cS_k$,
\begin{equation}\label{pref2}
	(A_k)_{jj}= \rho - a_{j}.
\end{equation}

\item \label{Ppref*}
  For the special type $ * $, we have 
  \begin{equation}\label{pref*}
(A_k)_{kk} =\chi.	
  \end{equation}
  \end{romenumerate}
Consequently,
\begin{equation}\label{trpref}
	\tr(A_k)= \chi+\rho - \sum_{j = 1}^{k-1} a_{j}.
\end{equation}
\end{prop}

\begin{proof}
This is similar to the proof of Proposition \ref{prop:gaps}.
Consider first a normal type $T_i$.
If we draw $T_i$, then we add a node to it, which either gives a new living
type $T_j\neq T_i$, or a tree with a dead root that is decomposed. In the
latter case, the only possibility to get a copy of $T_i$ is that only the
root is dead, and that when we remove it, the remaining $|T_i|$ living nodes
form a tree isomorphic to $T_i$, i.e., that we first add a node to $T_i$ and
then remove the root, and obtain a copy of $T_i$. 
In this exceptional case, the root of $T_i$ has to have degree at most 1,
and by induction it follows that every node in $T_i$ has out-degree at most
1, so $T_i$ is a path; furthermore, the new node has to be added at the leaf,
and $T_i$ has to have maximal length (since otherwise the root would not die
when the new node is added).

\pfitemref{Ppref1}
In this case, a drawn ball of type $T_i$ is never replaced by a
ball of the same type; thus $\xi_{ii}=-1$ and, by \eqref{A},
$(A_k)_{ii}=-a_i$, which yields \eqref{pref1}.

\pfitemref{Ppref2} 
If $T_j$ is a maximal path, then when adding a new node to $T_j$, the probability of
adding it at the leaf is $w_0/w_{T_j}$. In this
case, and only then, we obtain a new ball $T_j$. Consequently,
\begin{equation}
  \E \xi_{jj} = -1 + w_0/w_{T_j}=-1+w_0/a_j
\end{equation}
and, by \eqref{A},
\begin{equation}
(A_k)_{jj}= a_j\E\xi_{jj}
=a_j(-1+w_0/a_j)
=-a_j+w_0=-a_j+\rho,
\end{equation}
which yields \eqref{pref2}.

\pfitemref{Ppref*} 
The special type $*$ has activity 1, and 
if we draw a ball of type $* $, then
the total change in the urn is $ \chi $ additional balls of type $ * $ and
one additional ball of type $\bullet$;
hence
$(A_k)_{kk}= \E \xi_{kk}=\chi$.
\end{proof}

\begin{thm}
\label{thm:eigenvaluesrecursive}
For the linear preferential attachment trees, the eigenvalues of the
intensity matrix $A$ are $\chi+\rho $ and $- a_{i} $ for
$i\in\{1,\dots,q-1\} $, where   
$a_{i}$ is given by \eqref{tw2}.
\end{thm}

\begin{proof} 

The proof is  similar to the proof of Theorem \ref{thm:eigenvalues}.
Again we will use induction, and consider the
$k\times k$ matrices $A_k$ defined above.

We start the induction with $ k=2 $, when the types are \set{\bullet,*}. 
This means that all nodes with children are dead. If we draw a ball of type $\bullet$, 
we add a node to it and get a tree with two nodes; 
the root is dead and we obtain a
new ball of type $\bullet$ and a dead node of type $*_1$, which is changed to a mass 
$w_1=\chi+\rho$ of type $*$. A ball of type $\bullet$ is thus replaced by another
ball of type $\bullet$ and a mass $\chi+\rho$ of type $*$, so $\xi_{11}=1-1=0$ (in accordance with
\eqref{pref2}) and $\xi_{12}=\chi+\rho$. Similarly, 
$\xi_{21}=1$ and $\xi_{22}=\chi$. Consequently, by \eqref{A}, 
since $a_1=w_{T_1}=w_0=\rho$,
\begin{equation}\label{A2pref}
A_2=\begin{pmatrix}
0&1\\
\rho(\chi+\rho) & \chi
\end{pmatrix}.  
\end{equation}
The matrix $ A_2 $ has eigenvalues $ \chi+\rho $ and $ -\rho=-a_{1}$,
verifying the theorem for $A_2$.

The induction step is identical to the one in the 
proof of Theorem \ref{thm:eigenvalues}. We see again
that the eigenvalues of $ A_{k+1} $ are inherited from $ A_k $, i.e., that 
the eigenvalues of $A_{k+1}$ can be listed (with multiplicities) as
$\gl_1,\dots,\gl_{k+1}$ where 
$\gl_1,\dots,\gl_{k}$ where are the eigenvalues of $A_k$.
Finally, we use \eqref{trpref} and deduce
\begin{equation}
  \gl_{k+1}=\tr(A_{k+1})-\tr(A_k)
  =-a_{k},
\end{equation}
 and the theorem follows by induction.
\end{proof}

\begin{rem} For the linear preferential attachment tree the eigenvalues are
  easier to describe than for the $ m $-ary search tree. The reason for this
  is that we can start the induction already when we have a \Polya{} urn
  consisting of only two types. (The brave reader might even start with
  only one type, regarding all nodes as dead, and $A_1=(\chi+\rho)$.)
 In the $ m $-ary search tree, on the other hand, we always have at least $
 m-1 $ different types  as explained above, and these first types are the 
 reason why we get more  complicated eigenvalues.  
\end{rem}

\begin{proof}[Proof of \refT{recursivemulti}]

The proof is analogous to the proof of \refT{multivariate}.
By \refT{thm:eigenvaluesrecursive}, $\gl_1=\chi+\rho$ and all other eigenvalues
are negative.
Hence,  \refT{thm:normality}\ref{T0a} applies, so
\eqref{t0a} holds, with $\mu=\gl_1v_1$.
Furthermore, by \cite{JansonMean},
\begin{equation}\label{tollis2}
 \E\cX_n=n\mu+O(1).
\end{equation}

By an analogue of \eqref{xy}, 
$\bZ_{n}=\bigpar{X^{\Lambda^{1}}_{n},X^{\Lambda^{2}}_{n},\dots,X^{\Lambda^{d}}_{n}}=R\cX_n$
for some (explicit) linear operator $R$. Hence, \eqref{t0a} implies
\begin{equation}\label{dixi2}
  n\qqw\bigpar{\bZ_n-nR\mu}
=R\bigpar{n\qqw(\cX_n-\mu)}
\dto \N\bigpar{0,R\gS R'}
\end{equation}
and \eqref{tollis2} implies
\begin{equation}\label{vox2}
\bmu_n=\E\bZ_n=R\bigpar{\E\cX_n}=nR\mu+O(1),
\end{equation}
which together yield \eqref{multipref}, with the covariance matrix $R\gS R'$,
where $\gS$ is as in \eqref{t0a}.

Moreover, as said in \refR{branchjagerspreferential}, it follows from
\cite[(5.29)  and Example 6.4, in particular (6.24)]{HJbranching},
that
\begin{equation}\label{mensa2}
\E\bZ_n=n\hbmu+o(n).
\end{equation}
By \eqref{vox2} and \eqref{mensa2}, we have $R\mu=\hbmu$, 
and thus \eqref{vox2} yields \eqref{rpref}.

To see that the covariance matrix $R\gS R'$ is non-singular, suppose that, on
the contrary,
$u'R\gS R'u=0$ for some vector $u\neq0$. 
Then, by \cite[Theorem 3.6]{JansonMean},
$u'\bZ_n=u'R \cX_n$ is deterministic for every $n$.
However, this is impossible by the same construction as in the proof of 
\refT{multivariate}, taking $m=2$ (so the number of internal nodes equals
the number of keys) and then ignoring keys.
See also the proofs of 
\cite[Lemmas 3.6 and  3.17]{HolmgrenJanson}, which give this construction in
the special case of
the random recursive tree, and note that (at least for $\chi\ge0$),
the general case is the same since for a given $n$, the possible trees
$\gL_n$ are the same for all $\chi\ge0$ and $\rho$
(although their different probability distributions are different).
\end{proof}

\begin{proof}[Proof of \refT{variance2}]
Let $\Lambda^1,\dots,\Lambda^d$ be all non-random unordered (or ordered) trees with $k$ keys.
Then $Y_{n,k}=\sum_{i=1}^d X_n^{\Lambda^i}$. The result \eqref{prefksubtrees} thus
follows from \eqref{multipref}, and $\gs_k^2>0$
because the covariance matrix $\gS$ in \eqref{multipref} is non-singular.

Furthermore, summing \eqref{rpref} over the trees $\Lambda^i$  yields
\eqref{obliquuspref} 
and thus \eqref{rec} follows.
\end{proof}

\section{A \Polya{} urn to count protected nodes in $m$-ary search trees} \label{protectedpolya}
We start precisely in the same way as in \cite{HolmgrenJanson2}.
Given an $m$-ary search tree $\mathcal T_n$ with $n$ keys together with its external nodes, erase all edges that connect two internal non-leaves.  This yields a forest of small trees, where (assuming $n \ge m$) each tree has a root that is a non-leaf in $\mathcal T_n$. We regard these small trees as the balls in our generalised \Polya{} urn. The type of a ball (tree) is the type of the tree as an unordered tree, i.e., up to permutations of the children. The type of a tree in the urn is thus described by the numbers $x_i , i = 1, \dots, m$, of children of the root with $i - 1$
keys; each of these children is an external node ($i = 1$) or a leaf ($i \ge
2$), and it has itself children only when $i = m$ when it has $m$ external
children; thus we can unambiguously label the type by a vector $\xb = (x_1 ,
\dots, x_{m})$. Since the root of any of the small trees has $m$ children
(including external ones) in the original tree $\mathcal T_n$, we have $\sum_{i=1}^{m}
x_i \le m$, (with the remainder $m - \sum_{i = 1}^{m}x_i$ equal to the
number of erased edges to children in the original tree $\mathcal T_n$ that are
non-leaves). If $n \ge m$, there are no balls of type $(m, 0, \dots, 0)$,
since the root of a small tree is never a leaf in $\mathcal T_n$, and therefore the
total number of types is the number of ways to write $m$ as a sum of
$m+1$ non-negative integers minus one, i.e., $\binom{2m} m-1$.   

The activity of a type $\xb$ is the number of gaps it contains. The root has
no gaps and each child with $i - 1$ keys contributes $i$ gaps.
(These gaps belong to the child itself when $i\le m-1$ and to the $m$
external children of it when $i=m$.)
Therefore type
$\xb$ has activity $a_\xb := \sum_{i=1}^{m} i x_i$.  

Let us denote the unit vectors
\begin{equation*}
	\ee_1 = (1, 0, \dots, 0), \quad \ee_2 = (0,1,0,\dots,0), \quad \dots, \quad \ee_{m} = (0, \dots, 0, 1).
\end{equation*}
If we add a new key to a leaf with $i - 1 \le m-2$ keys, which belongs to a small tree of type $\xb$, in the \Polya{} urn this corresponds to replacing the ball by a ball of type $\xb - \ee_{i} + \ee_{i+1}$. The same holds if we add a key to an external node that is a child of the root. However, if we add a key to an external node that is a child of a (full) leaf in a tree of type $\xb$, then that leaf becomes a non-leaf, so the edge from it to the root is erased and the tree is split into two trees: one of type $(m-1, 1, 0, \dots 0)$ and the other of type $\xb - \ee_{m}$. Thus in general there may be up to $m$ different ways a small tree can be transformed, depending on which gap a new key goes into. 

\section{Proof of Theorem \ref{mainprotected}}\label{protectedmaryproof}
The idea of the proof is the same as in the case of fringe subtrees:
reducing the urn one type at a time, until we arrive at the simple urn for
counting incomplete leaves with matrix $A_{m-1}$. In addition to the types
described in Section \ref{protectedpolya}, add types
numbered~$1$ to $m$ where type $i$ is a single node with $i - 1$ keys
for $i\le m-1$ and type $m$ is the type $(m,0,\dots,0)$, \ie, a node with
$m-1$ keys and $m$ external children.
Of course, when $n \ge m$, there are no small trees of these types, but they are
used in the induction.

Recalling that the number of types described above is $\binom {2m} m-1$ and writing $q = m-1+\binom {2m} m$, let us enumerate all types from $1$ to $q$ in such a way that whenever we insert a key into a tree of type $i$ (and obtain one or two trees, as described above), the new tree which inherits the root of the old one must have index larger than $i$. Note the similarity with the fringe case where the tree grown by inserting a key must have a larger index.

There is more than one such enumeration of types, but for clarity we choose the
following one. We let the first $m-1$ types be the single nodes ordered with
an increasing number of keys, while the rest of the types are ordered in such a
way that a type with more children precedes a type with fewer
children; ordering among types with equal number of children is achieved
by treating them as numbers in base $m + 1$ (with the coordinate $x_1$ being
the most significant digit) arranged in decreasing order.  
(For types with the same number of children, this is the reverse
lexicographic order.) The vector describing a type $i \ge m$ is denoted by $\xb^{(i)}$.
Thus, for example, type $m$ is type $\xb^{(m)} = (m, 0, \dots, 0)$ and is followed by
$\xb^{(m+1)} = (m-1,1,0, \dots, 0)$, and the last type is the dead type $\xb^{(q)} = (0, \dots, 0)$.  
Figure \ref{ternarytypes} shows the different types in a ternary search tree
(i.e., $ m=3 $), except for the first 3, which are single nodes (and are the
same as the first three in Figure \ref{trinitypes}).

Let us write $\cX_n = (X_{n,1}, \dots, X_{n,q})$, where $X_{n,i}$ is the number of balls of type $i$ in the urn. The limit distribution of the urn is, of course, described by types $m+1$ to $q$, since for $n \ge m$ we have no balls of the first $m$ types.

For each $k = m - 1, \dots, q$ we define a reduced urn by considering the forest of small trees in the original urn and deleting roots of the trees of types $i > k$ (and the edges leading to their children). This replaces each tree of type $i$ by $\xb^{(i)}_j$ trees of type $j$, for each $j = 1, \dots, m$. Write
\begin{equation*}
	\cX^k_n = (X^k_{n,1}, \dots, X^k_{n,k}).
\end{equation*}
Clearly, there is a $k \times q$ matrix $P_k$ such that $\cX^k_n = P_k \cX_n$ and therefore
\begin{equation}\label{eq:DeltaX}
  \Delta \cX^k_n = \cX^k_{n+1} - \cX^k_n = P_k \left( \Delta \cX_n \right).
\end{equation}
\begin{figure}
\includegraphics[scale=0.103]{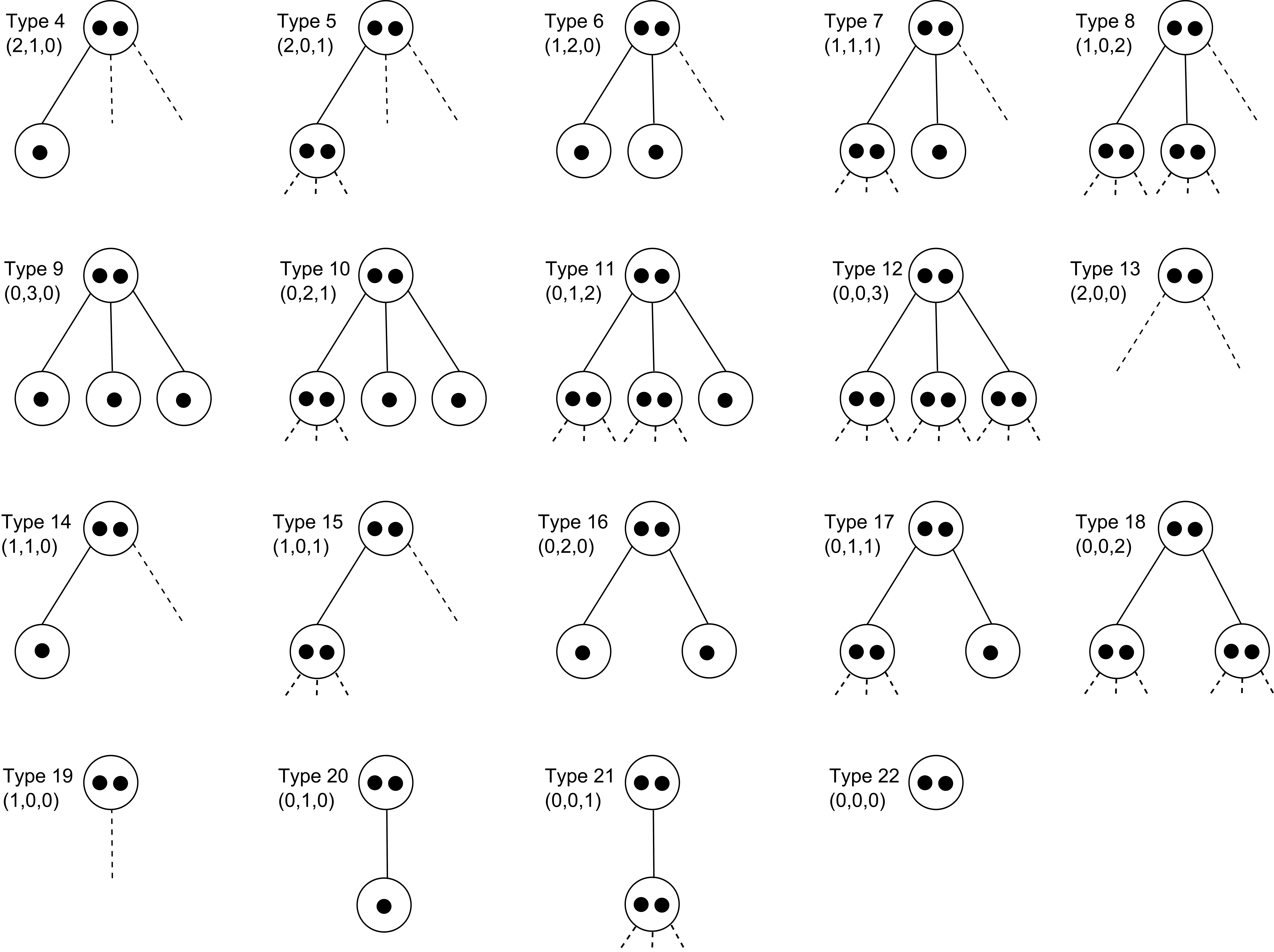}
\caption{The 19 different types characterizing protected and unprotected
  nodes in ternary search trees for $ n\geq 3 $. Type 13, type 19 and type
  22 are the only ones that include a protected node.}\label{ternarytypes}
\end{figure}

It need not be obvious that $\cX^k_n$ is actually a \Polya{} urn, in the sense that the distribution of $\Delta \cX^k_n$ depends only on the type $i$ of the ball drawn. To convince ourselves we consider three cases.

\textbf{Case 1:} If we draw a ball of type $i \in \left\{ m+1, \dots, k
\right\}$, then in the original urn this corresponds to drawing a ball of
the same type. Therefore, in view of \eqref{eq:DeltaX}, $\Delta \cX^k_n$ has
the same distribution as $P_k \xi_i$, which depends on $i$ only.

\textbf{Case 2:} If we draw a ball of type $i < m$, then in the original urn (assuming $n \ge m$) this corresponds to replacing a tree of some type $r > k$ (for which $x^{(r)}_i > 0$) by a tree of type $\xb^{(s)} := \xb^{(r)} - \ee_i + \ee_{i+1}$. The latter tree inherits the root of the former, which, by the ordering of types, implies that $s > r > k$. Therefore in the reduced urn a ball of type $i$ is replaced by a ball of type $i+1$ (with an exception when $k = m-1$ and $i = m-1$, in which case it is replaced by $m$ balls of type $1$), regardless of which particular type $r$ was involved.

\textbf{Case 3:} If we draw a ball of type $i = m$, then in the original urn
(assuming $n \ge m$) we replace a tree of type $\xb^{(r)}$ (for which
$x^{(r)}_m > 0$) by a tree of type $m+1$ and a tree of type $\xb^{(s)} := \xb^{(r)} - \ee_m$, which inherits the root and therefore satisfies $s > r > k$. In the reduced urn we remove a ball of type $m$ and what we add depends on the value of $k$: if $k \ge m +1$, we add a ball of type $m+1$; otherwise $k = m$ and we add $m$ balls: $m-1$ external nodes and one node with a single key. In either case the outcome does not depend on $r$.

To sum up, we have shown that $\cX_n^k$, $k = m-1, \dots, q$, are \Polya{}
urns with some random replacement vectors $\xi^{(k)}_i$, $i = 1, \dots,
k$. Regardless 
of $k$, the activity of a type $i$ is the number of gaps in the
corresponding tree, which we denote $a_i$. So $a_i = i$ for $i = 1, \dots,
m-1$ and $a_i = a_{\xb^{(i)}}$ for $i \ge m$. Hence the intensity matrices of
the reduced urns are  
\begin{equation*}
	A_k := (a_j\E \xi^{(k)}_{ji})_{i,j=1}^k.
\end{equation*}
As in the fringe case, the diagonal values of $A_k$ for $m \ge 3$ are easy to determine.
\begin{prop} \label{prop:protdiag}
For $m \ge 3$, $k = m-1, \dots, q$ and every type $i = 1, \dots, k$ we have $(A_k)_{ii} = -a_i$.
\end{prop}
\begin{proof} We need to show that whenever a ball of type $i$ is drawn, it is never replaced.
	For $i \le m$ this is clear from the discussion above (cases 2 and 3), so let us assume $i > m$. Suppose that a key is inserted into a child with less than $m-1$ keys. Then in the original urn a ball of type $i$ is replaced by a ball of higher index $j$, so in the reduced urn $i$ is replaced with either a ball of type $j$ or some balls with types at most $m$, none of which coincide with type $i$. If a key is inserted into a full child of the root (which means $x^{(i)}_m > 0$), then also a ball of type $\xb^{(m+1)} = (m-1,1,0,\dots,0)$ is added. Since its last coordinate is zero, it cannot coincide with type $i$.
\end{proof}

The number of protected nodes is the total number of balls of types $(i, 0, \dots, 0)$, $i = 0, \dots, m-1$.
The urn $\cX^{q-1}_n$ ignores the dead type $(0, \dots, 0)$. However, writing
$b_i$ for the number of keys in a tree of type $i$, the number of balls of
the dead type can be expressed as an affine function of the other types, namely
\begin{equation*}
	X_{n,q} = \frac{1}{m-1} \left( n - \sum_{i=m+1}^{q-1} b_i X_{n,i}\right).
\end{equation*}
Consider the urn $\cX_n^{q - m - 1}$ obtained by deleting the roots of trees
with at most one child. For $n \ge m$, the first $m$ coordinates of
$\cX_n^{q-m-1}$ (single nodes and the type $(m, 0, \dots, 0)$) are equal to
the last $m$ coordinates of $\cX_n^{q - 1}$ (roots with one
child). Consequently, the urn described in \refS{protectedpolya} but with
the dead type ignored, and thus $\binom{2m}m-2=q-m-1$ types, is isomorphic
to the urn $\cX_n^{q-m-1}$ and has thus also intensity matrix $A_{q-m-1}$ (up
to relabelling the types).
In particular, the number of
protected nodes is an affine function of $\cX_n^{q- m -  1}$ and therefore, 
in order to prove asymptotic normality of the number of
protected nodes, it is enough to prove asymptotic normality of 
$\cX_n^{q - m  -1}$. To deduce this from Theorem \ref{thm:normality},  
it remains to show that $\Re\lambda<\lambda_1/2 $
for each eigenvalue $ \lambda\neq \lambda_1 $, 
which, for $m\le26$, follows from the next
theorem.
\begin{thm}
\label{thm:proteigenvalues}
Let $m \ge 2$. The eigenvalues of $A_{q-m-1}$ are the roots of the
polynomial $\phi_m(\lambda) = \prod_{i = 1}^{m-1} (\lambda + i) - m!$ plus
the multiset 
\begin{equation*}
\left\{ -a_i : i = m, m + 1, \dots, q - m - 1 \right\}.
\end{equation*} 
\end{thm}
\begin{proof}
For  $m = 2$, the eigenvalues were determined in
\cite{HolmgrenJanson2} 
as \set{1,-2,-3,-4} (when we ignore the dead type), which agrees with 
the statement.
	
For $m \ge 3$ the proof goes along the same lines as the proof of Theorem
\ref{thm:eigenvalues}; we again show by induction the corresponding
statement for $A_k$ for $m-1\le k\le q-m-1$.
Note that urn $\cX_n^{m-1}$ has the same meaning as
in the fringe case, since $X_{n,i}^{m-1}$, $i = 1, \dots, m-1$, stands for
the number of nodes with $i-1$ keys. Therefore $A_{m-1}$ is defined by
\eqref{AWmatrix} and has characteristic polynomial $\phi_m(\gl)$.
	
We use Proposition \ref{prop:protdiag} instead of Proposition
\ref{prop:gaps}. The linear map $T$ such that $\cX_n^k  = T\cX_n^{k+1}$ is a
$(k+1) \times k   $ matrix obtained by appending to an identity matrix a
column $(x_1, \dots, x_m, 0, \dots, 0)'$, where $\xb^{(k+1)} = (x_1, \dots,
x_m)$ is the vector describing type $k+~1$. The rest of the proof is
identical. 
\end{proof}

\begin{proof}[Proof of \refT{mainprotected}]
This is similar to the other proofs. 
As in \refS{proofs}, for $m\le26$, the roots $\gl\neq\gl_1$ of $\phi_m$ satisfy
 $\Re\gl\le\gam_m<\frac12=\gl_1/2$, 
where $\gam_m$ is given by \eqref{gamm},
and thus \refT{thm:proteigenvalues}  shows that
\eqref{gammm} holds
for all eigenvalues $\gl\neq\gl_1$ of $A_{q-m-1}$.
As said above, $A_{q-m-1}$ is also the intensity matrix of the urn in
\refS{protectedpolya}, without the dead type,
and if we reinstate the dead type, we just add one eigenvalue 0 (since the
new column in $A$ is identically 0). 
Hence \refT{thm:normality} applies to the urn in \refS{protectedpolya}, and
implies,
since $Z_{n}$ is the total number of balls of the $m$ types $(i,0,\dots,0)$,
$0\le i\le m-1$, 
\begin{equation}
  n\qqw\bigpar{Z_n-\mu'n}\dto \cN(0,\gss)
\end{equation}
for some $\mu'$ and $\gss\ge0$.
Furthermore, it follows from \cite{JansonMean} and \eqref{gammm}
that 
(for any $m$)
\begin{equation}\label{muprot1}
\E(Z_{n})=\mu'n+O\bigpar{n^{\gammm}}.
\end{equation}
Moreover, it follows from
\cite[Theorems 7.11, 10.1 and 10.3]{HJbranching} that
$Z_{n}/n\asto \mu_m$ given by \eqref{muprot},
and consequently, by dominated convergence 
(see also \cite[Remark 5.19]{HJbranching})
\begin{equation}\label{muprot0}
\E(Z_{n})=\mu n+o(n).
\end{equation}
By \eqref{muprot1} and \eqref{muprot0}, $\mu'=\mu$ (for any $m$).
When
$m\le26$, we have
$\gam_m<\frac12$, and thus  \eqref{muprot1} implies \eqref{muprot0qq},
and thus $\E Z_{n}$ can be replaced by $\mu' n=\mu n$ in
\eqref{prot}. 

To see that $\gss_m>0$, it suffices by \cite[Theorem 3.6]{JansonMean} to
show that $\Var(Z_{n})>0$ for some $n$. This is easy; for example, with
$n=2m-1$ keys, we can have either 0 or 1 protected node.
\end{proof}

\section{Examples with explicit calculations of variances}\label{examples}

We give some examples with explicit calculations (done using Mathematica).

\subsection{Example of Theorem \ref{variance} when $ m=3 $ and $ k=4 $}\label{ex1}

We consider the case when we want to evaluate $ \sigma_4^{2} $ in Theorem
\ref{variance} in the case of a random ternary search tree ($m=3$).
We use the construction of the \Polya{} urn in  \refE{fringektrees}, which
gives an urn with the following 6 different (living) types:
{\addtolength{\leftmargini}{-10pt}
\begin{enumerate}
 \renewcommand{\labelenumi}{\textup{\arabic{enumi}:}}%
 \renewcommand{\theenumi}{\textup{\arabic{enumi}}}%
\item 
 An empty node.  
\item 
A node with one key.
\item 
 A node with two keys and three external children. 
\item 
A tree with a root holding two keys and one child holding one key, plus two
external children. 
\item 
 A tree with a root holding two keys and two
children holding one key each, plus one external child. 
\item 
 A tree with a root holding two keys and one child holding two keys, 
plus two external children of the root and three external children of the leaf.
\end{enumerate}}
See Figure \ref{trinitypes} for an illustration of these types.

\begin{figure}
\includegraphics[scale=0.6]{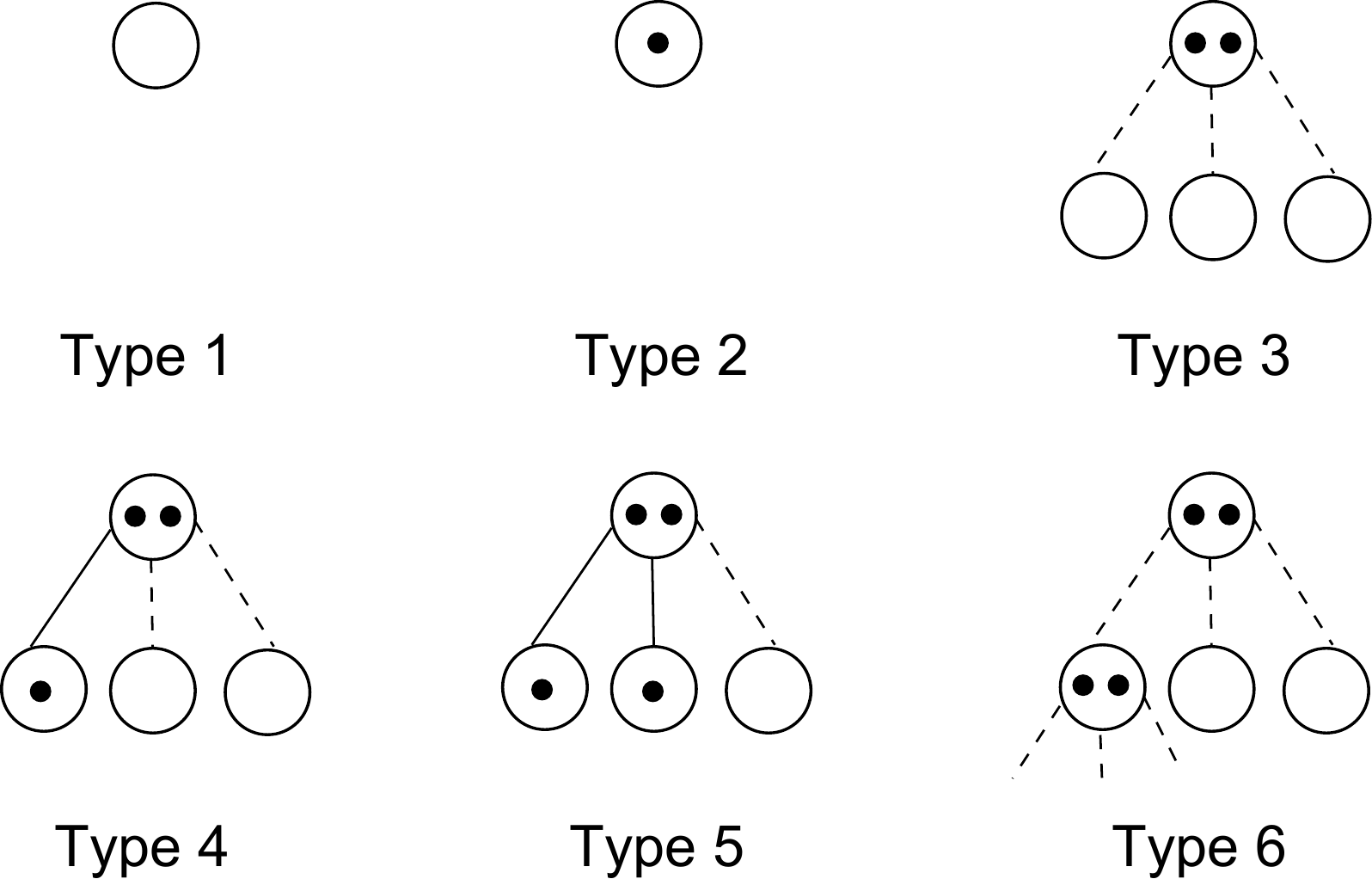}
\caption{The different types for counting the number of the fringe subtrees with four keys in a ternary  search tree.}\label{trinitypes}
\end{figure} 

The activities of the types are $1,2,3,4,5,5$. We can easily describe the intensity matrix, first noting that if we draw a
type $ k $  for $ k\leq 3 $ it is replaced by one of type $ k+1 $. If we
draw a type 4 it is replaced by one of type 5 with probability 1/2 and one
of type 6 with probability 1/2. If we draw a type 5 it is replaced by three
of type 2 with probability 1/5, and one of each of the types 1, 2 and 3 with
probability 4/5; see Figure \ref{trinitrans1} for an illustration.  Finally
if we draw a type 6 it is replaced by one each of 
the types 1, 2 and 3 with probability 2/5, and two of type 1 and one of type
4 with probability 3/5; see Figure \ref{trinitrans2} for an illustration.  

\begin{figure}
\includegraphics[scale=0.5]{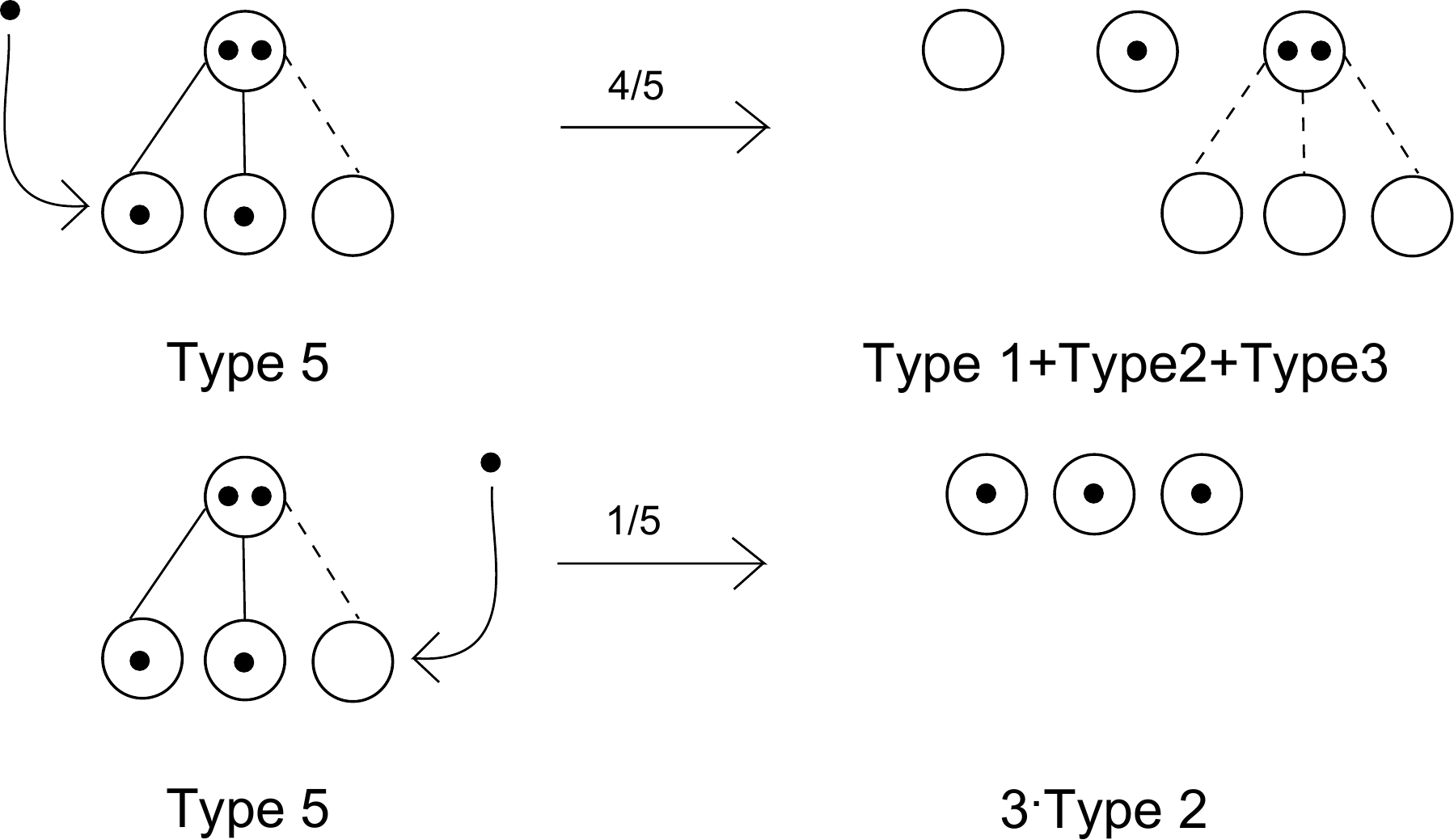}
\caption{The two possibilities for adding a key to 
a tree of type 5
in Figure \ref{trinitypes}.}
\label{trinitrans1}
\end{figure}

\begin{figure}
\includegraphics[scale=0.5]{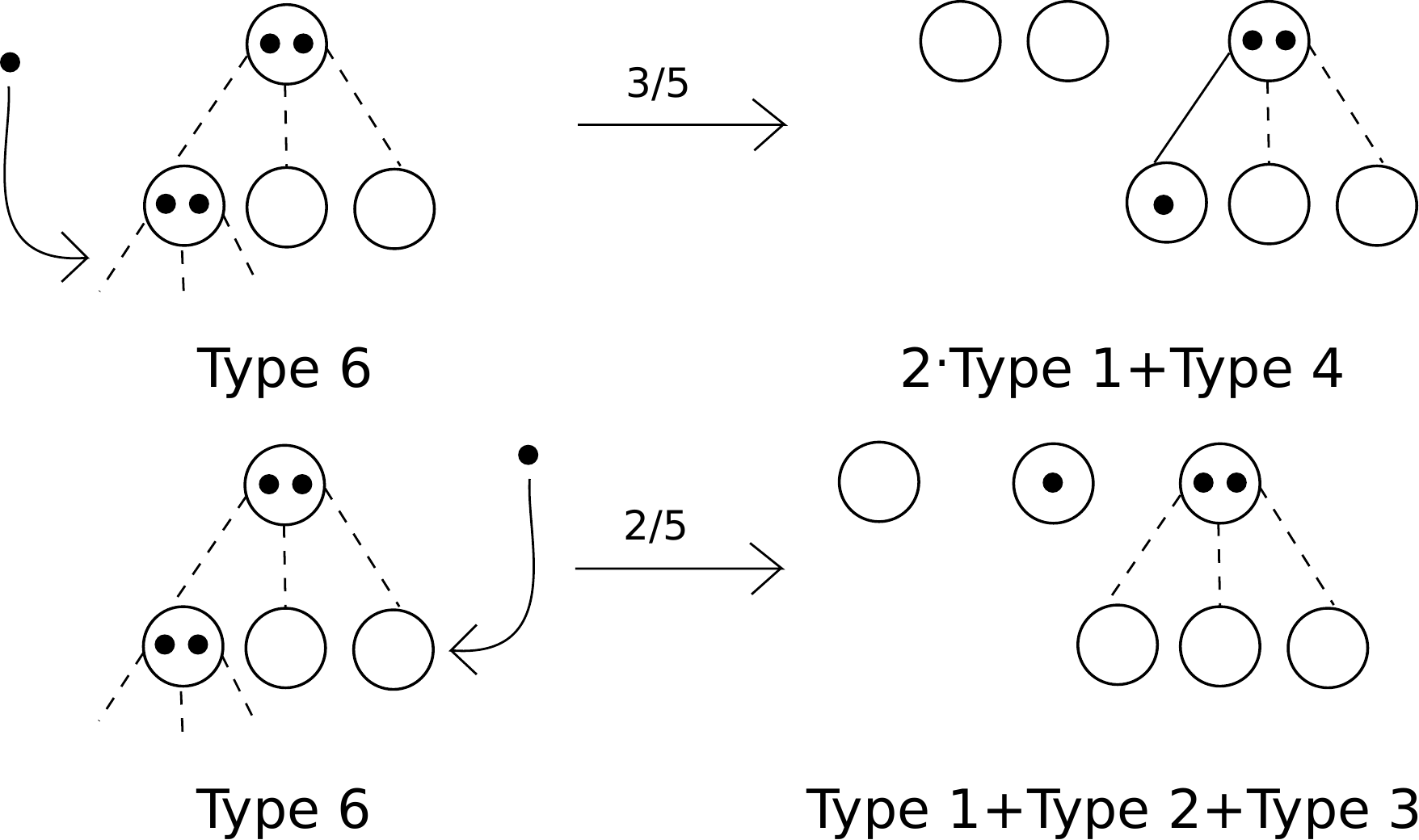}
\caption{The two possibilities for adding a key to 
a tree of type 6
in Figure \ref{trinitypes}}
\label{trinitrans2}
\end{figure}
     
Thus, we get the intensity matrix $ A $ in (\ref{A}) as
\begin{equation}
  A= \left(
\begin{array}{rrrrrr}
 -1 & 0 & 0 & 0 & 4 & 8 \\
 1 & -2 & 0 & 0 & 7 & 2 \\
 0 & 2 & -3 & 0 & 4 & 2 \\
 0 & 0 & 3 & -4 & 0 & 3 \\
 0 & 0 & 0 & 2 & -5 & 0 \\
 0 & 0 & 0 & 2 & 0 & -5 \\
\end{array}
\right).  
\end{equation}
The eigenvalues are,
by direct calculation or by \refT{thm:eigenvalues}, 
\begin{equation}
1, -3, -4, -4,  -5, -5.
\end{equation}
(We know already that $\gl_1=1$, as was noted in \refR{Racdotmary}.)

Furthermore, by \refR{Racdotmary},
the left eigenvector $u_1= a=(1, 2, 3, 4, 5, 5) $.
The right eigenvector $v_1$, with the normalization \eqref{normalised}, is 
found to be
\begin{equation}\label{eq:v1}
 v_1=\Bigpar{\frac{3}{25}, \frac{1}{10}, \frac{2}{25}, \frac{3}{50},
   \frac{1}{50}, \frac{1}{50}}.   
\end{equation}
Note that $\mu=v_1$ in the proof of \refT{multivariate} (of which Theorem \ref{variance} is a direct consequence), since $\gl_1=1$.
The fringe subtrees with 4 keys in the random ternary search tree
correspond to the last two types, so $Y_{n,4}=X^{T^5}_n+X_n^{T^6}$.
Hence, the expected number of subtrees with 4 keys is, see \eqref{tollis},
\begin{equation}
  \E Y_{n,4} = (\mu_5+\mu_6)n+o\bigpar{n\qq}= \frac{1}{25}n+o\bigpar{n\qq}. 
\end{equation}
 Note that this gives the same answer as Theorem \ref{variance}
(where results from branching processes were applied to deduce the
expectation), since the asymptotic expectation in \eqref{main2b} is
$$ \frac{n}{(H_3-1)(4+1)(5+1)}=\frac{n}{25}.$$ 

To calculate the variance $ \sigma^2_4 $, we
calculate the covariance matrix $\gS$ in Theorem \ref{thm:normality} 
by Theorem \ref{thm:normality}\ref{T0b}; thus we first
calculate $ B_i $, $B$ and $\Sigma_I $ in (\ref{Bi})--(\ref{Sigma}).
We  describe the calculations briefly;
for further details on how the calculations are performed, 
we refer to \cite{HolmgrenJanson2}, where
Theorem \ref{thm:normality}\ref{T0b} was applied to the urn in
Figure \ref{ternarytypes} in \refS{protectedmaryproof} above
to count the number of protected nodes in a random ternary search tree. 
Since $A$ is diagonalisable,  it is, as an alternative,  also possible to 
calculate $\Sigma$ by Theorem \ref{thm:normality}\ref{T0c}.

We thus first calculate  
$ B_i=\E(\xi_i\xi_i') $ in (\ref{Bi}). 
As an example we have (the other cases are analogous)
\begin{align}
B_5&=\tfrac{1}5\cdot b_1b_1'+\tfrac{4}5\cdot b_2b_2',
\intertext{where}
   b_1&=(0, 3, 0, 0, -1, 0)'
\qquad \text{and}
\qquad
b_2=(1, 1, 1, 0, -1, 0)'. 
\end{align}
%
We then calculate $B$ by \eqref{B} and evaluate the integral in
(\ref{Sigma}), which 
yields $\Sigma_I$. Finally, $\Sigma=\Sigma_I$ by 
Theorem \ref{thm:normality} with $c=1$.
The result is 
 \begin{align}\label{covariance4tree}
 \arraycolsep=3pt
\def\arraystretch{2}
 \Sigma=\left(
\begin{array}{rrrrrr}
 \frac{29017}{259875} & -\frac{117371}{10395000} & -\frac{44311}{5197500} &
   -\frac{2143}{945000} & -\frac{28289}{5197500} & -\frac{28289}{5197500} \\
 -\frac{117371}{10395000} & \frac{7379}{83160} & -\frac{34927}{5197500} &
   -\frac{3907}{236250} & -\frac{166037}{20790000} & -\frac{166037}{20790000} \\
 -\frac{44311}{5197500} & -\frac{34927}{5197500} & \frac{159241}{2598750} &
   -\frac{4747}{236250} & -\frac{84709}{10395000} & -\frac{84709}{10395000} \\
   -\frac{2143}{945000} & -\frac{3907}{236250} & -\frac{4747}{236250} &
   \frac{39227}{945000} & -\frac{13309}{1890000} & -\frac{13309}{1890000} \\
 -\frac{28289}{5197500} & -\frac{166037}{20790000} & -\frac{84709}{10395000} &
   -\frac{13309}{1890000} & \frac{22613}{1299375} & -\frac{6749}{2598750} \\
 -\frac{28289}{5197500} & -\frac{166037}{20790000} & -\frac{84709}{10395000} &
   -\frac{13309}{1890000} & -\frac{6749}{2598750} & \frac{22613}{1299375} \\
\end{array}
\right).\end{align}

However, to calculate $ \sigma^{2}_4 $, we only need the submatrix
\begin{align}\label{covariancesub4}
\Delta=\left(
\begin{array}{rr}
 \sigma_{5,5} &  \sigma_{5,6} \\[10pt]
  \sigma_{6,5} &  \sigma_{6,6}
\end{array}
\right)=
\left(
\begin{array}{rr}
 \frac{22613}{1299375} & -\frac{6749}{2598750} \\[10pt]
 -\frac{6749}{2598750} & \frac{22613}{1299375}
\end{array}
\right).
\end{align}
Summing the $\sigma_{i,j}  $  in \eqref{covariancesub4}, 
which is equivalent to calculating $ (1,1)\Delta(1,1)' $, we find
  $$\sigma^2_4=\frac{38477}{1299375}.$$

 Note that we can use this urn to calculate the asymptotic variance for the
 number of leaves in the random ternary search tree, which was evaluated
 in \cite[Theorem 4.1]{HolmgrenJanson2}. 
 We get   
 \begin{equation}
(0, 1, 1, 1, 2, 1)\Sigma (0, 1, 1, 1, 2, 1)'=\frac{89}{2100}.   
 \end{equation}
 We could also use this urn to evaluate 
 \begin{align}
\sigma^2_3&=(0, 0, 0, 1, 0, 0)\Sigma(0, 0, 0, 1, 0, 0)'=\frac{39227}{945000}, 
\\
\sigma^2_2&=(0, 0, 1, 0, 0, 1)\Sigma(0, 0, 1, 0, 0, 1)'= \frac{131}{2100},
\\
\sigma^2_1&=(0, 1, 0, 1, 2, 0)\Sigma(0, 1, 0, 1, 2, 0)'= \frac{8}{75}.
 \end{align}

\subsection{Example of Theorem \ref{multivariate} when $ m=4 $} \label{ex2}
 
We consider the random quaternary search tree ($ m=4 $) 
as an ordered tree.
Suppose that we consider two fringe subtrees in this tree. Let the first one
$T^1$ consist of 5 keys i.e., $k_1=5$, so that the root holds three keys
with its first two children holding one key each, and with its remaining two
external children to the right. Let the second one $T^2$ consist of 6 keys
i.e., $k_2=6$, so that the root holds three keys, and with its first child
holding three keys and thus also having four external children, and with the
remaining three external children of the root to the right. 

 We use the construction of the \Polya{} urn in  \refS{polyafringe} which
gives an urn with the following 9 different (living) types,
 see Figure \ref{quadritypes}:
{\addtolength{\leftmargini}{-10pt}
\begin{description}
\item [\rm1--4]
For $k\le4$, type $k$ is a node with $ k-1 $ keys; the fourth type also has
four external children.   
\item [\rm5]
A root with three keys with its four children having $(1,0,0,0)$ keys.
\item [\rm6]
A root with three keys with its four children having $(0,1,0,0)$ keys.
\item [\rm7] 
A root with three keys with its four children having $(1,1,0,0)$ keys.
\item [\rm8]
A root with three keys with its four children having $(2,0,0,0)$ keys.
\item[\rm9]
A root with three keys with its four children having $(3,0,0,0)$ keys, and
the first child having four external children.
\end{description}}
     
 \begin{figure}
\includegraphics[scale=0.5]{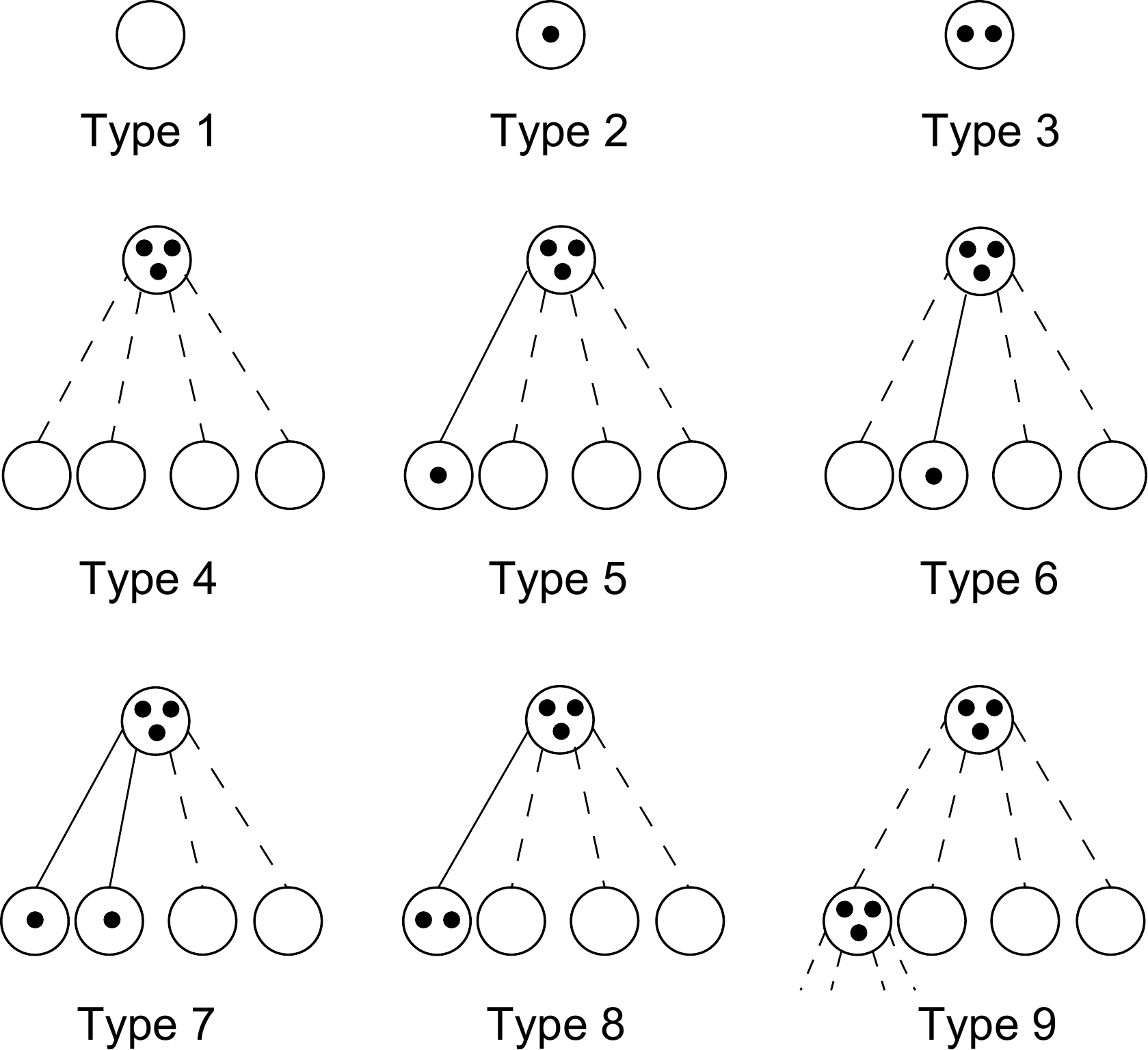}
\caption{The different types used to find the joint distribution of fringe
  trees that are isomorphic to type 7 and type 9 in a quaternary search
  tree.}
\label{quadritypes}
\end{figure}  
  
We get the intensity matrix  as in the example in Section \ref{ex1}.
We describe one example of a transition, the others are similar. If we draw a type 5 it is replaced by one of type 7 with probability 1/5, and one of type 8 with probability 2/5, and two of type 1 and  two of type 2 with probability 2/5; see Figure \ref{quadritrans}.

\begin{figure}
\includegraphics[scale=0.4]{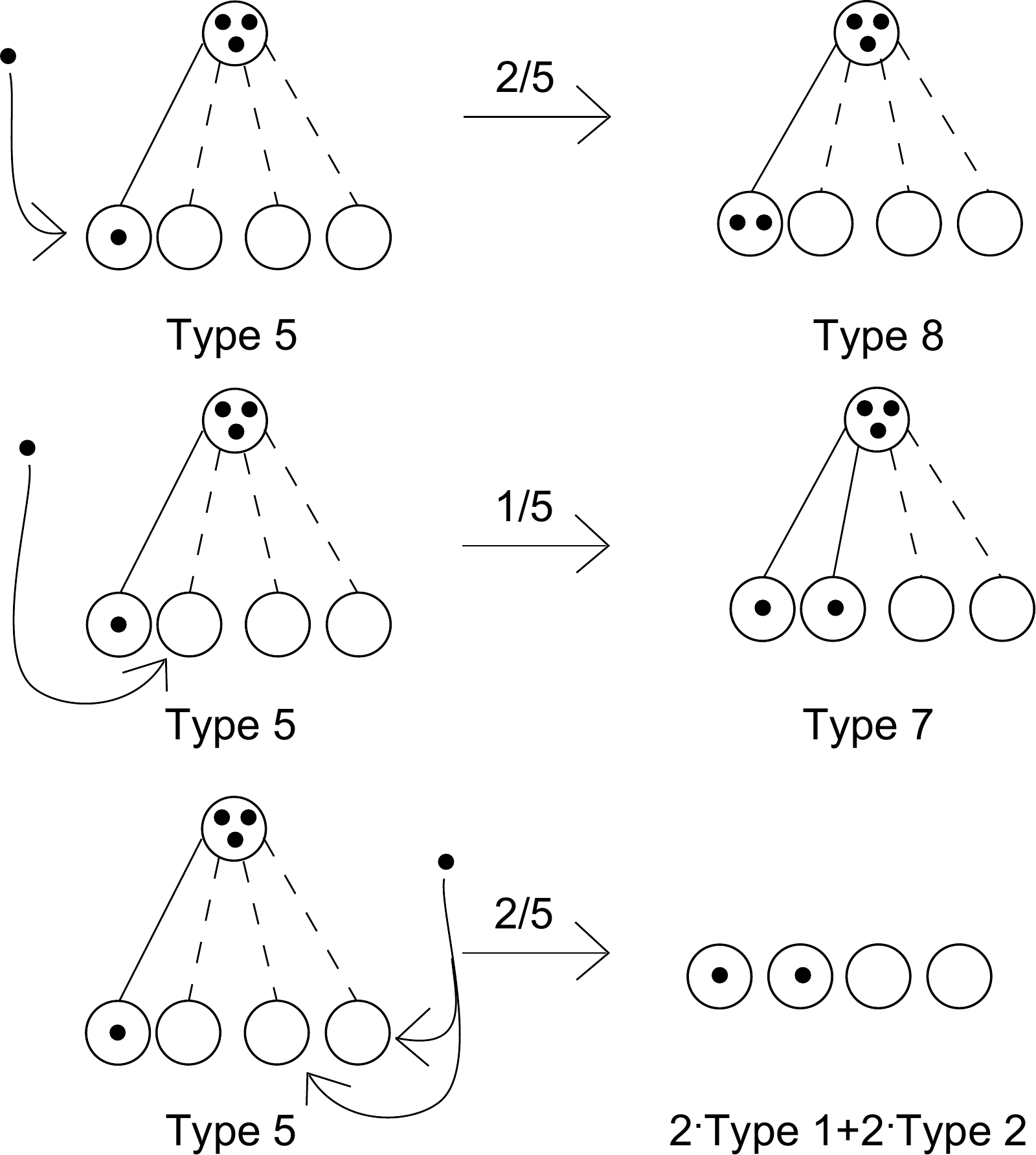}
\caption{The three possibilities for adding a key to 
a tree of type 5
in Figure \ref{quadritypes}.}
\label{quadritrans}
\end{figure}


The intensity matrix is 
\begin{align}\label{Atern} 
A=\left(
\begin{array}{rrrrrrrrr}
 -1 & 0 & 0 & 6 & 4 & 10 & 10 & 6 & 24 \\
 1 & -2 & 0 & 2 & 4 & 4 & 10 & 3 & 5 \\
 0 & 2 & -3 & 0 & 0 & 2 & 4 & 3 & 0 \\
 0 & 0 & 3 & -4 & 0 & 0 & 0 & 0 & 3 \\
 0 & 0 & 0 & 1 & -5 & 0 & 0 & 0 & 1 \\
 0 & 0 & 0 & 1 & 0 & -5 & 0 & 0 & 1 \\
 0 & 0 & 0 & 0 & 1 & 1 & -6 & 0 & 0 \\
 0 & 0 & 0 & 0 & 2 & 0 & 0 & -6 & 0 \\
 0 & 0 & 0 & 0 & 0 & 0 & 0 & 3 & -7 \\
\end{array}
\right).
\end{align} 

The eigenvalues are,
by direct calculation or by \refT{thm:eigenvalues}, 
\begin{equation}
1, -\tfrac72 + \tfrac{\sqrt{23}}2i, -\tfrac72 - \tfrac{\sqrt{23}}2i,
-4,-5, -5,-6, -6, -7.  
\end{equation}
Note that the activity vector $ a=(1, 2, 3, 4, 5, 5, 6, 6, 7) $ and it again
follows that the left eigenvector $ u_1=a $.
The right eigenvector $v_1$,
with the normalization 
\eqref{normalised}, is calculated as
\begin{equation}
  \label{v1quattor}
v_1=\Bigpar{\frac{113}{520}, \frac{61}{455}, \frac{34}{455}, \frac{33}{728}, \frac{1}{130}, \frac{1}{130}, \frac{1}{455}, \frac{1}{455}, \frac{3}{3640}}.
\end{equation}
We see from this vector that the expected number of fringe subtrees
 of the two types that we consider is 
$ (\frac{1}{455} +\frac{3}{3640})n+o(n)=\frac{11}{3640}n+o(n) $, 
since the two types are type 7 and type 9.
We can verify that this gives the same answer as Theorem \ref{multivariate}
(where results from branching processes were applied to deduce the
expectation), since the vector $\hbmu$  in \eqref{hmu} has coordinates
\begin{align}
\frac{\P(\mathcal{T}_{5}=T^{1})}{(H_4-1)(5+1)(5+2)} &=
\frac{\frac35\cdot\frac24\cdot\frac13}{(\frac12+\frac13+\frac14)\cdot6\cdot7}
=\frac{1}{455}
\intertext{and}
 \frac{ \P(\mathcal{T}_{6}=T^{2})}{(H_4-1)(6+1)(6+2)}&=
\frac{\frac36\cdot\frac25\cdot\frac14}{(\frac12+\frac13+\frac14)\cdot7\cdot8}
=\frac{3}{3640}.
\end{align}

We proceed by calculating the covariance matrix $ \Sigma $.
It again turns out that $A$ is diagonalisable, and this time we apply Theorem
\ref{thm:normality}\ref{T0c}. We 
again have to  calculate the matrix $ B $ in (\ref{B}), and for this 
we have to calculate  
$ B_i=\E(\xi_i\xi_i') $ in (\ref{Bi}). 
For example, 
\begin{align}
B_7&=\tfrac{4}6\cdot b_1b_1'+\tfrac{2}6\cdot b_2b_2',
\intertext{where}
   b_1&=(2, 1, 1, 0, 0, 0, -1, 0, 0)'
\qquad\text{and}\qquad
b_2=(1, 3, 0, 0, 0, 0,-1, 0, 0)'. 
\end{align}
Having calculated $B_1,\dots,B_9$ in this way, we obtain the matrix
$B$ from \eqref{B}
and then  the covariance matrix $ \Sigma $ by \eqref{simpleSigma}; 
the result is shown in Appendix \ref{appA}.  
However, to evaluate the joint distribution of the fringe subtrees $ T^1 $ and $ T^{2} $ described above, we only need the submatrix
\begin{align}
\Gamma=\left(
\begin{array}{cc}
 \sigma_{7,7} &  \sigma_{7,9} \\[10pt]
  \sigma_{9,7} &  \sigma_{9,9}
\end{array}
\right)=
\left(
\begin{array}{cc}
 \phantom- \frac{157523}{72872800} & -\frac{2884319}{194424630400} \\[10pt]
 -\frac{2884319}{194424630400} & \phantom- \frac{5681341}{6943736800} 
\end{array}
\right).
\end{align}

Note that we can also use this urn to calculate the asymptotic variance
$\sigma_1^2$  for 
  the number of leaves in the random quaternary search tree;
 we get   
\begin{align}\label{leaf4}
\sigma_1^2 = (0, 1, 1, 1, 1, 1, 2, 1, 1)\Sigma (0, 1, 1, 1, 1, 1, 2, 1, 1)'=\frac{5276}{122525}.
 \end{align}
We can alternatively obtain  $\sigma_1^2$ by using the simple \Polya{} urn with the four types illustrated in Figure \ref{quadrileaf}. 
(As another simple example of Theorem \ref{multivariate} and the
construction in \refS{polyafringe},  
see also \cite[Section 5.2]{HolmgrenJanson2} for a minor variation of this
urn.)  
\begin{figure}
\includegraphics[scale=0.6]{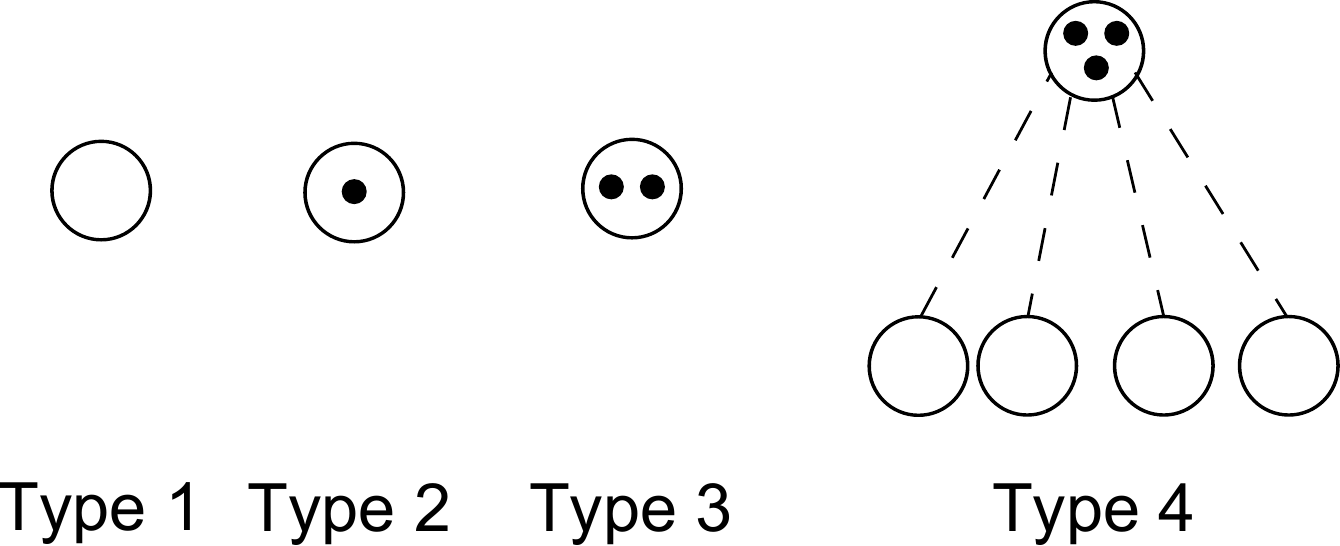}
\caption{The different types used to study the number of leaves in a
  quaternary search tree.}
\label{quadrileaf}
\end{figure}
The intensity matrix is 
\begin{equation}
A=\left(
\begin{array}{rrrr}
 -1 & 0 & 0 & 12 \\
 1 & -2 & 0 & 4 \\
 0 & 2 & -3 & 0 \\
 0 & 0 & 3 & -4 \\
\end{array}
\right). 
\end{equation}
The eigenvalues are,
by direct calculation or by \refT{thm:eigenvalues}, 
\begin{equation}
1, -\tfrac72 + \tfrac{\sqrt{23}}2i, -\tfrac72 - \tfrac{\sqrt{23}}2i, -4.
\end{equation}
Since $ A $ is diagonalisable we may again use Theorem
\ref{thm:normality}\ref{T0c} to calculate $ \Sigma $. 
We obtain 
\begin{align}
  \def\arraystretch{1.5}
\Sigma=\left(
\begin{array}{rrrr}
 \frac{34466}{122525} & \frac{153}{49010} & -\frac{963}{24505} &
   -\frac{10393}{245050} \\
 \frac{153}{49010} & \frac{519}{4901} & -\frac{339}{9802} & -\frac{681}{24505}
   \\
 -\frac{963}{24505} & -\frac{339}{9802} & \frac{276}{4901} & -\frac{57}{3770} \\
 -\frac{10393}{245050} & -\frac{681}{24505} & -\frac{57}{3770} &
   \frac{4391}{122525} \\
\end{array}
\right).
\end{align}
From $ \Sigma $ we see that the asymptotic variance for the number of leaves 
in a random quaternary search tree is 
 \begin{align}\label{simpleleaf4}(0, 1, 1, 1)\Sigma (0, 1, 1, 1)'=\frac{5276}{122525},
 \end{align}
which equals the result in (\ref{leaf4}).

\subsection{Example of Theorem \ref{variance2} when $ k=3 $}\label{ex3}

We consider the case when we want to evaluate $ \sigma_3^{2} $ in Theorem \ref{variance2} in the case of a linear preferential attachment tree. 
As explained in Section \ref{recursive}, only the quotient $ \chi/\rho $
matters, and thus we may assume that $\chi\in\set{-1,0,1}$. 
For (notational) simplicity, we consider only the case $\chi=1$ in the
formulas below; the cases $\chi=0,-1$ (and general $\chi$) are similar but
are left to the reader. The final results will be expressed in the parameter
$\gk=\rho/(\chi+\rho)$ 
in \eqref{kappa}; these results are valid for all values of $\chi$ and
$\rho$. (This follows either by checking the other cases, or by analytic
continuation, since the eigenvector $v_1$ and the integral \eqref{Sigma} are
analytic functions of $(\chi,\rho)$ in a suitable domain.)

We use the construction of the \Polya{} urn in  
\refS{orderpref}, with $\cS'$ the set of (unordered) trees with at most 3
nodes. (Cf.\ \refE{fringektrees} for \mst{s}.)
This
gives an urn with the following 5 different types,
see Figure \ref{preftypes}:
{\addtolength{\leftmargini}{-10pt}
\begin{enumerate}
 \renewcommand{\labelenumi}{\textup{\arabic{enumi}:}}%
 \renewcommand{\theenumi}{\textup{\arabic{enumi}}}%
\item 
 An empty node.  
\item 
A path with two nodes.
\item 
 A path with three nodes. 
\item 
A tree consisting of one root with two children.
\item 
 A special type (with activity 1). 
\end{enumerate}}

\begin{figure}
\includegraphics[scale=0.6]{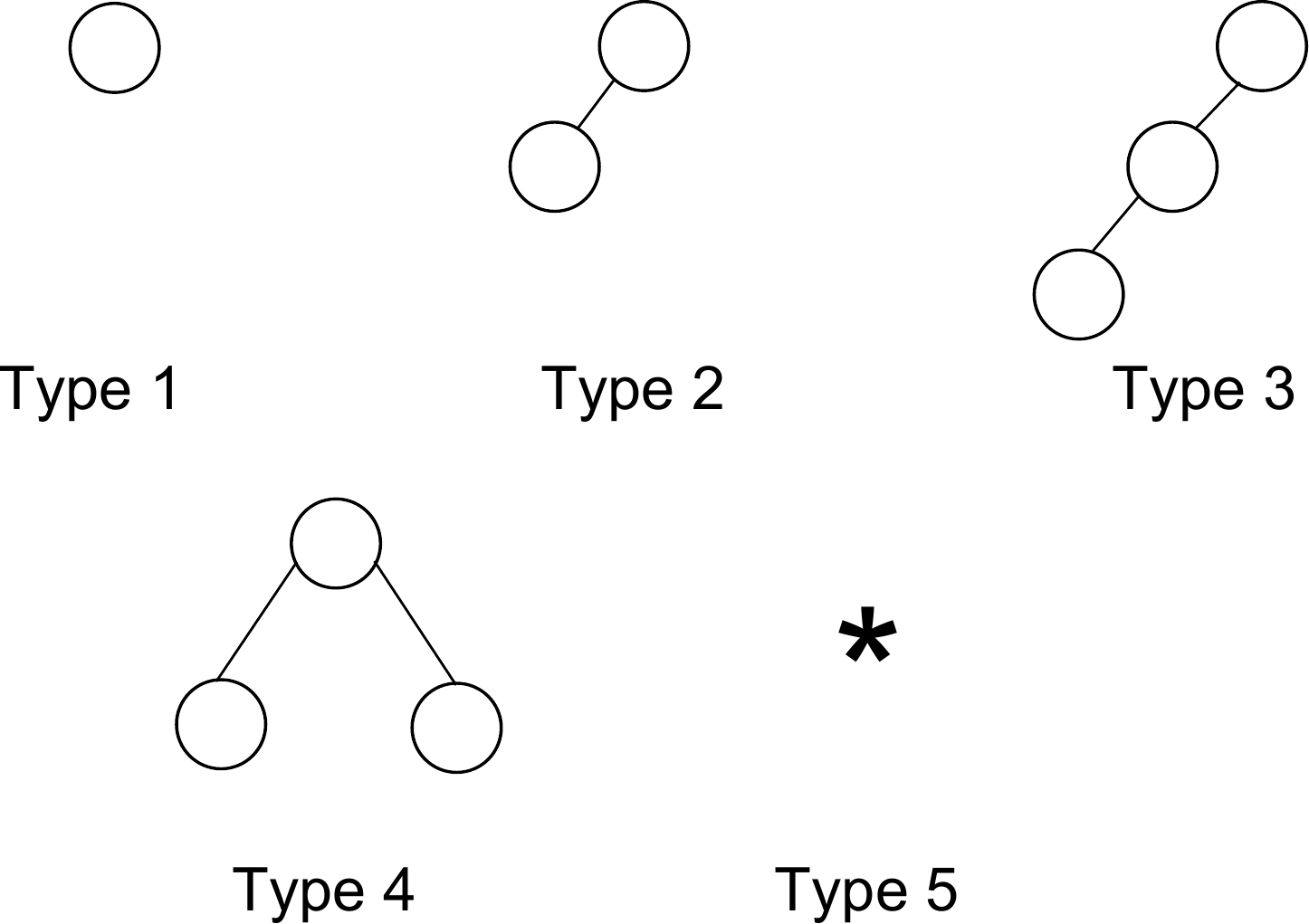}
\caption{The different types used to study the number of fringe subtrees
  with three nodes in a preferential attachment tree.}
\label{preftypes}
\end{figure}

We get the intensity matrix as in the examples  in Section
\ref{ex1} and Section \ref{ex2}. 
We describe one example of a transition, the others are similar. If we draw
a type 3 it is 
replaced by, 
see Figure \ref{preftrans},
\begin{itemize}
\item 
1 of type 1, 1 of type 2, and $ 2+\rho $ of type 5
with probability $ \frac{1+\rho}{2+3\rho} $;
\item 
1 of type 4 and  $ 1+\rho $ of type 5
with probability $ \frac{1+\rho}{2+3\rho} $;
\item 
1 of type 3 and $ 1+\rho $ of type 5
with probability $\frac{\rho}{2+3\rho}  $.
\end{itemize}

\begin{figure}
\includegraphics[scale=0.5]{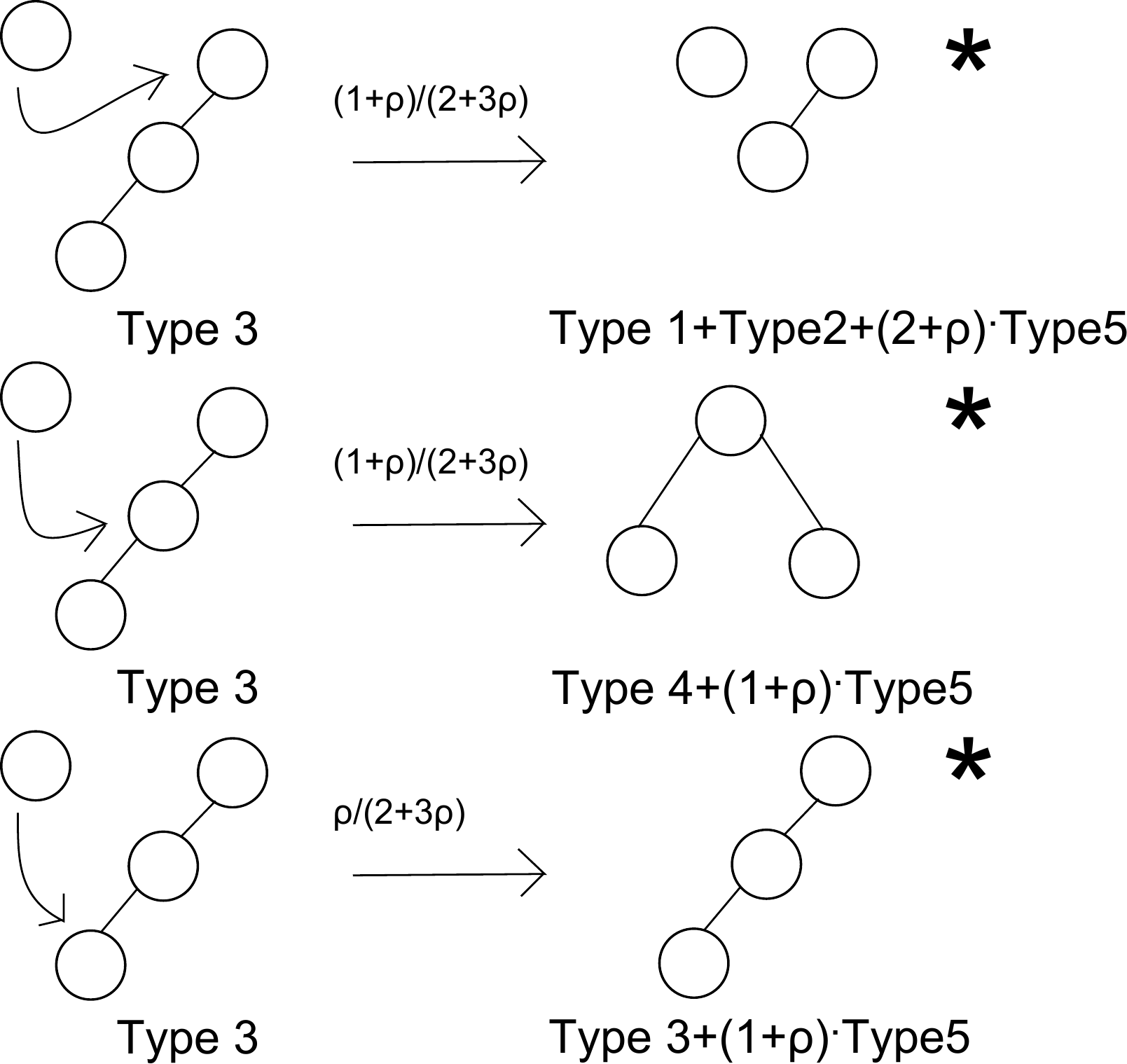}
\caption{The three possibilities for adding an additional node to 
  a tree of type 3 
in Figure \ref{preftypes}.}
\label{preftrans}
\end{figure}

The intensity matrix is 
\begin{align}\label{Apref}
 \def\arraystretch{1.25}
 A=\left(
\begin{array}{ccccc}
 -\rho & 0 & \rho+1 & 5 \rho+6 & 1 \\
 \rho & -2 \rho-1 & \rho+1 & 2 \rho & 0 \\
 0 & \rho & -2 \rho-2 & 0 & 0 \\
 0 & \rho+1 & \rho+1 & -3 \rho-2 & 0 \\
 0 & 0 & 3 (\rho+1)^2 & 3 (\rho+1) (\rho+2) & 1 \\
\end{array}
\right).
\end{align}

The eigenvalues are,
simplest
by \refT{thm:eigenvaluesrecursive}, 
$$\rho+1,-\rho,-2 \rho-1,-3 \rho-2,-3 \rho-2.$$
It again turns out that $A$ is diagonalisable.
To calculate $ \Sigma $ we apply Theorem
\ref{thm:normality}\ref{T0c}. We again have to first calculate 
 $ B_i $ and  $B$  in (\ref{Bi})--(\ref{B}).

Note that the activity vector $ a=(\rho, 1 + 2 \rho, 2 + 3 \rho, 2 + 3 \rho,
1) $ and it again follows that the left eigenvector $ u_1=a $.  
The eigenvector $v_1$ (with the normalisation \eqref{normalised})
is calculated to be
\begin{multline} \label{v1rho}
v_1= \frac{1}{(2\rho + 1)(3\rho + 2)(4\rho + 3)}\Bigpar{6 (\rho+1)^2, 3 \rho (\rho+1), \rho^2, \rho (\rho+1), 6 (\rho+1)^3}.
   \end{multline}
We express this using 
$\gk$ 
in \eqref{kappa}, which yields 
\begin{multline}\label{v1kappa}
v_1 = \frac{1}{(\kappa + 1)(\kappa + 2)(\kappa + 3)}
\Bigpar{
6(1-\gk), 3 (1-\gk) \gk, (1-\gk) \gk^2, (1-\gk) \gk, 6}.
\end{multline}

Recall from the proof of \refT{recursivemulti} (of which Theorem \ref{variance2} is a direct consequence) that $\mu = \gl_1v_1$,
where $\gl_1=\rho+1=1/(1-\gk)$.
The fringe subtrees with three nodes correspond to type 3 and type 4,
so $Y_{n,3}=X_{n,3} + X_{n,4}$.
Hence, from \eqref{tollis2} and \eqref{v1kappa}  
we obtain 
\begin{equation}
\E Y_{n,3} = 
\frac{1}{1-\gk}
\,\frac{(1-\gk) \gk^2+(1-\gk)\gk}{\gkn}
\,n+O(1)
=\frac{\gk}{(\gk+2)(\gk+3)}n+O(1),
\end{equation}
which agrees with
Theorem \ref{variance2} and \eqref{obliquuspref}.

We next calculate
$ B_i=\E(\xi_i\xi_i') $ in (\ref{Bi}); we take
$B_3 $ as an example, where we get,
see Figure \ref{preftrans},
\begin{align}
B_3&=\tfrac{1+\rho}{2+3\rho}\cdot b_1b_1'
+\tfrac{1+\rho}{2+3\rho}\cdot b_2b_2'+\tfrac{\rho}{2+3\rho}\cdot
b_3b_3',
\intertext{with}
 b_1&=(1,1,-1,0,\rho+2)',
   \\  
b_2&=(0,0,-1,1,\rho+1)',
\\
b_3&=(0,0,0,0,\rho+1)'. 
\end{align}
We then find the matrix $B$.
Finally, the eigenvectors $u_i$ and $v_i$ are calculated for all eigenvalues
and the covariance matrix $ \Sigma $ is calculated by \eqref{simpleSigma}.
The covariances are listed in  Appendix \ref{appB}, again using the notion $
\gk $ in  \eqref{kappa}.
As said above, these formulas are valid for all $\chi$ and $\rho$.

Returning to the number of fringe subtrees of order 3 we 
thus obtain
\begin{equation}\label{sigpref}
	\begin{split}
\sigma_3^2 &=(0, 0, 1, 1, 0)\Sigma(0, 0, 1, 1, 0)'
=\gS_{3,3}+2\gS_{3,4}+\gS_{4,4}
\\&
=\frac{3 \gk \left(11 \gk^3+52 \gk^2+77 \gk+30\right)}{2 (\gk+2) (\gk+3)^2 (2 \gk+1) (2 \gk+3) (2
   \gk+5)}.	  
	\end{split}
\end{equation}
We can check that this formula yields previously known results (obtained by
other methods) in the three most important special cases, Examples
\ref{ERRT}--\ref{EBST}. 

For the random recursive tree,  $ \gk=1 $ and 
\eqref{sigpref} yields
$  \sigma_3^2 =\xfrac{17}{336}$, which equals the result
given by
\citet[Theorem 4]{Devroye1},
where $\sigma_k^2 $ was calculated for general $k$ (using different methods),
see also \cite{Fuchs} and \cite[Proposition 1.13 and (1.20)]{HolmgrenJanson}.

For the binary search tree, $\gk=2$ and 
\eqref{sigpref} yields
$\gs_3^2=8/175$, 
which agrees with the result by 
\citet[Theorem 5]{Devroye1} (for general $k$), 
see also
\cite[Proposition 1.10]{HolmgrenJanson}.

For the plane oriented recursive tree, $\gk=1/2$ and
\eqref{sigpref} yields
$\sigma_3^2 = 663/15680 $. 
This variance was calculated, for general $k$, 
by \citet[Theorem 1.1]{Fuchs2012} by other methods (generating functions). 

\begin{rem}\label{RFuchs}
There is a mistake in the
formula for 
the asymptotic variance in \cite[Theorem 1.1]{Fuchs2012}:
the numerator $8k^2-4k-8$ should be $8k^2-4k$.
(The reason is that in the
calculation 
of $\Var(X_{n,k})$ on \cite[p.~419]{Fuchs2012}, 
there should be  a plus sign in front of $\frac{4}{(4k^{2}-1)^{2}}$
instead of a minus sign.)
With this correction, \eqref{sigpref} (with $\gk=1/2$) agrees with
the value for $ k=3 $ of the formula in
\cite[Theorem 1.1]{Fuchs2012}, and the values for $\gss_1$ and $\gss_2$
obtained below agree with the values of the formula for $k=1,2$.
\end{rem}

Note that we can use the \Polya{} urn in this example to calculate
also $\gss_1$ and $\gss_2$, \ie{} the constants in the
asymptotic variances for the numbers $Y_{n,1}$ and $Y_{n,2}$
of fringe subtrees of size $1$ and $2$ in a  
linear preferential attachment tree. 
Note that $Y_{n,1}$ is simply the number of leaves, \ie, the number of nodes
of out-degree 0.
We have $Y_{n,1}=X_{n,1}+X_{n,2}+X_{n,3}+2X_{n,4}$ and 
$Y_{n,2}=X_{n,2}+X_{n,3}$, see Figure \ref{preftypes}.
Thus, again using Appendix \ref{appB},
\begin{align}
\gss_1&
=(1, 1, 1, 2, 0)\Sigma (1, 1, 1, 2, 0)'
=\frac{\gk}{(\gk+1)^2 (2 \gk+1)}, \label{gss1pref}
\\
\gss_2&
=(0, 1, 1, 0, 0)\Sigma (0, 1, 1, 0, 0)'
= \frac{\gk \left(5 \gk^2+10 \gk+6\right)}{(\gk+1) (\gk+2)^2 (2 \gk+1) (2 \gk+3)}.\label{gss2pref}
 \end{align}
However, these values can also be obtained by smaller urns, yielding simpler
calculations, for examples 
by the urn with only types 1, 2 and 5 (the special type) above.
For $k=1$ it suffices to use the urn with two colours and intensity matrix
\eqref{A2pref} used 
in the proof of \refT{thm:eigenvaluesrecursive}, see
\refE{Eleaves}.

We can again check that the results agrees 
with known results when $\gk=1,\frac12,2$.
For the random recursive tree ($\gk=1$),
\eqref{gss1pref}--\eqref{gss2pref} yield
$\gss_1=1/12$ and $\gss_2=7/90$, as shown in 
\cite[Theorem 4]{Devroye1}, 
see also \cite[Proposition 1.13]{HolmgrenJanson};
for the plane oriented recursive tree ($\gk=1/2$) we obtain
$\gss_1=1/9$ and $\gss_2=49/600$, as shown in 
\cite{MahmoudSS} and
\cite[Theorem 1.1]{Fuchs2012} (see \refR{RFuchs});
for the binary search tree ($\gk=2$) we obtain
$\gss_1=2/45$ and $\gss_2=23/420$, as shown in 
\cite[Theorem 5]{Devroye1}, 
see also \cite{Fuchs} and 
\cite[Proposition 1.10]{HolmgrenJanson}.

\section{Degrees}\label{degree}

By using (simpler) variants of the \Polya{} urns described above for
studying fringe subtrees in $m$-ary search trees and  preferential
attachment trees, we can also easily prove normal limit theorems for the
out-degrees of the nodes  in both of these models.  

\subsection{Out-degrees in $ m $-ary search trees}
We first consider \mst{s}. 
The following theorem was recently proved by \citet{KalMahmoud:Degree} using
a  \Polya{} urn (based on gaps instead of nodes) that is equivalent to the
one used here.
(A simpler version, $A_m$ in the proof below, 
was used in \cite{HolmgrenJanson2} to study the special
case $k=0$, \ie, the number of leaves.)
We nevertheless sketch a proof 
in order to show the connections with the analysis in the previous sections,
and in particular the induction argument for the eigenvalues. 
(In \cite{KalMahmoud:Degree}, the eigenvalues were calculated numerically.)

\begin{thm}\label{degreemary} 
Let $D_{n,k}$ be the number of nodes with out-degree $k$ in the random $ m $-ary
search tree $ \cT_n $. 
If $m\leq 26$, then,  as $n\to \infty$,  
 \begin{align}\label{degm}
\dfrac{D_{n,k}-\mu_k n}{\sqrt{n}}
&\dto \mathcal{N}(0,\sigma^{2}_k), 
 \end{align}
with
\begin{equation}\label{mstD}
\mu_k=
\begin{cases}
\frac{m-1}{2(H_m-1)(m+1)}, & k=0,
\\
\frac{1}{(H_m-1)m(m+1)}, & 1\le k\le m,
\end{cases}
\end{equation}
 where 
$\sigma^{2}_k$ is some positive constant.
 \end{thm}

\begin{proof}
We construct a \Polya{} urn by chopping up the \mst{} as in
\refS{polyafringe},
but we now
erase all edges from parents to internal children, keeping only the edges to
external children. Hence, the small trees in the resulting forest, which are
represented by balls in the urn, are of the following $2m-1$ types.
(We regard the trees as unordered.)
\begin{itemize}
\item 
A single internal node with $1,\dots,m-2$ keys.
\item 
A root with $m-1$ keys and $0,\dots,m$ external children.
\end{itemize}
We can simplify a little by noting that the type consisting of a root with 0
external children is dead (activity 0) so it does not affect the evolution
of the urn and can be ignored (although it should be included in the final
count of node degrees; it represents the nodes of degree $m$).
Moreover, the type with 1 external child can instead be represented by a single
external node (although it should be counted as a node of degree $m-1$).
This yields a \Polya{} urn with the following $2m-2$ types.
(See Figure \ref{figdegree}, which shows the different types in the case
$ m=4 $.) 
\begin{figure}
\includegraphics[scale=0.6]{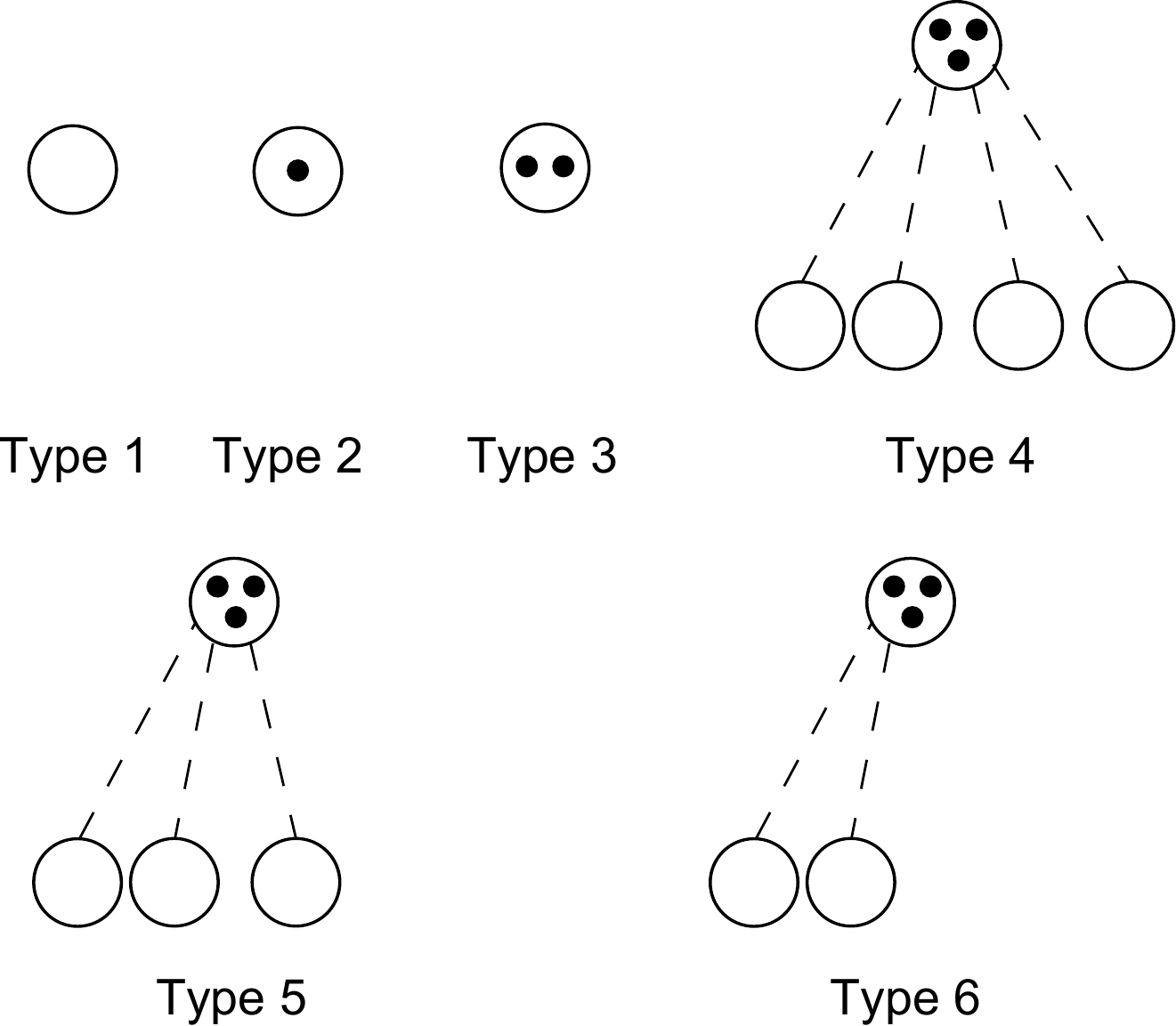}
\caption{The different types used to characterize the different out-degrees
  in a quaternary search tree.}
\label{figdegree}
\end{figure}
\begin{description}
\item [\rm$1,\dots,m-1$]
Type $i$ is a single node with $i-1$ keys. There are $i$ gaps and thus the
activity $a_i = i$.
\item [\rm$m,\dots,2m-2$]
Type $i$ is a node with $m-1$ keys and $2m-i$ external children. The
activity $a_i = 2m-i$. (The out-degree is $i-m$.)
\end{description}
If we draw a ball of type $i\le m-2$, it is replaced by a ball of type $i+1$.
Similarly, a ball of type $m-1$ is replaced by a ball of type $m$.
A ball of type $i\in\set{m,\dots,2m-3}$  
is replaced by a ball of type $i+1$ and a ball of
type 2.
A ball of type $2m-2$ is replaced by a ball of type 1 and a ball of type 2.

For our induction argument, we also consider a reduced urn with types
$1,\dots,k$, for $m - 1\le k\le 2m-2$, obtained by chopping up also all trees of
types $j>k$. In other words, we replace a ball of type $j>k$ by $2m-j$ balls
of type 1. (Just as we already have cut up trees with a single external
child.) This reduced urn thus ignores all nodes 
with $m-1$ keys and degree $>k-m$. 

Let $A_k$ be the intensity matrix
of this urn with $k$ types. When a ball is drawn, there will never be a ball
of the same type in the set of balls replacing it.
Hence every $\xi_{ii}=-1$ and
the diagonal
elements of $A_k$ are $-a_i$, where $a_i$ is the activity of type $i$.
For $k=m-1$, the reduced urn is the same as the urn with the $m-1$ first
types in \refS{proofs}, so $A_{m-1}$ is given by \eqref{AWmatrix} and its
eigenvalues are the roots of 
$\phi_m(\gl):= \prod_{i = 1}^{m-1} (\lambda + i) - m!$.
It now follows by the same
induction argument as in \refS{proofs} that the eigenvalues of $A_k$ are
the roots of $\phi_m$ plus \set{-m,-(m-1),\dots,-(2m-k)}. In particular, taking
$k=2m-2$, the eigenvalues of the intensity matrix $A=A_{2m-2}$ are the
eigenvalues of $\phi_m$ plus the negative numbers $-2,\dots,-m$.

Hence,  for every eigenvalue $\gl\neq\gl_1=1$,
$\Re\gl\le\max\xpar{\gam_m,-2}\le\gammm$, \cf{} \eqref{gamm},
and thus \refT{thm:normality} applies when $m\le26$.
The rest of the proof is as in the proof of \refT{multivariate} and
other proofs above. 
The constants \eqref{mstD} can be found either by finding the eigenvector
$v_1$ of the intensity matrix $A$ explicitly, as in \cite{KalMahmoud:Degree},
or by 
comparison with
\cite[Theorems 7.11 and 7.14]{HJbranching}
(proved using branching processes). 

Finally, $\gss_k>0$ by another application of \cite[Theorem 3.6]{JansonMean}
and the fact that $D_{n,k}$ is not deterministic for all $n$, as is easily
seen.
\end{proof}

\subsection{Out-degrees in preferential attachment trees}

\citet{MahmoudSmythe} used a \Polya{} urn to show
asymptotic normality for the numbers of nodes of out-degrees 0, 1 and 2
in a random recursive tree, and
\citet{MahmoudSS} did the same for a plane oriented recursive tree;
these results were extended to arbitrary degrees 
by \citet{SJ155}.
We can extend this to general linear preferential attachment trees (using
essentially the same urn).

In the case $\chi<0$, when we can assume $\chi=-1$ and $\rho=m$ for some
integer $m$, the resulting tree has no nodes of out-degree $>m$, and we consider
only out-degrees $k\le m$ in the following theorem; otherwise $k$ is
arbitrary. 
For the asymptotic proportions $\mu_k$ in \eqref{mstDpref}, see also 
\cite[(6.33)]{HJbranching}.

\begin{thm}\label{degreepref} 
Let $ \hD_{n,k} $ be number of nodes with out-degree $k$ in the in the linear
preferential attachment tree $ \Lambda_n $ defined by the weights
\eqref{wlinear}. 
Then,  as $n\to \infty$, 
 \begin{align}\label{main2bpref}
\dfrac{\hD_{n,k}-\mu_{k} n}{\sqrt{n}}
&\dto \mathcal{N}(0,\sigma^{2}_k), 
 \end{align}
with some $\gss_k>0$ and
\begin{equation}\label{mstDpref}
\mu_{k}
=\frac{w_1}{w_0+w_1}
\prod_{i=1}^{k}\frac{w_{i-1}}{w_i+w_1}
=\frac{\chi+\rho}{\chi +2\rho}
\prod_{i=1}^{k}\frac{\chi (i-1)+\rho}{\chi(i+1)+2\rho}.
\end{equation}
 \end{thm}
\begin{proof}
As in \refS{prefpolya}, we begin by constructing an urn with infinitely many
types. The types are \set{0,1,2,\dots}, and a ball of type $i$ simply
represents a node of out-degree $i$.
The activity of  type $i$ is thus $a_{i}= w_i=\chi i+\rho $, see
\eqref{wlinear}. When a ball of type $i$ is drawn, it is replaced by a ball
of type $i+1$ and a ball of type $0$.
(In the case $\chi<0$, we consider only a finite number of types
so we have a finite urn;
we leave the minor modifications in this case to the reader.)

To get an urn with finitely many types,
we truncate as in \refS{orderpref} (and \cite{SJ155});
we choose an integer $k\geq 0$ and use the $k+2$
types $\{0,1,\dots,k\}\cup\set{*}$, 
where the new type $*$ represents the activity of the nodes with
out-degrees $> k$. Hence, $*$ has activity 1. 
If we draw a ball of type $i<k$, we replace it by a ball of type $i+1$ and 
a ball of type 0, as before;
if we draw a ball of type $k$, we replace it by a ball of type 0
and $w_{k+1}=(k+1)\chi+\rho$ balls of type $*$;
if we draw a ball of type $*$, it is replaced 
and we add $\chi$ additional balls of type $*$ and a ball of type 0.

Let $A_{k+2}$ denote the $(k+2)\times(k+2)$ intensity matrix of this urn.
For $k=0$ we have the same urn as in the proof of
\refT{thm:eigenvaluesrecursive}, so $A_2$ is given by \eqref{A2pref} with
the eigenvalues $\chi+\rho$ and $-\rho$.
As in the proof of \refT{thm:eigenvalues}, 
the eigenvalues of $A_k$ are inherited by $A_{k+1}$, and a simple induction
shows that the eigenvalues of $A_{k+2}$ are
$\chi+\rho$ and $-w_j=-(\chi j+\rho)$ for $0\leq j\leq k$.  
In particular, all eigenvalues except $\gl_1=\chi+\rho$ are negative.
Hence \refT{thm:normality} applies and the proof is completed as the other
proofs. The constants $\mu_k$ can be found by verifying directly that
\eqref{mstDpref} yields an eigenvector of the intensity matrix $A$ for the
infinite urn, and thus it is mapped to an eigenvector for $A_{k+2}$ by the
truncation above. Alternatively, we can use \cite[(6.14)]{HJbranching}. 
\end{proof}

\begin{rem}
We can modify the urn in the proof of \refT{degreepref}
to an equivalent one
by changing each ball of type $i\le k$ to $w_i$ new
balls of type $i$; the new balls all have activity 1 and can be interpreted as
gaps. (This is the urn actually used in \cite{SJ155}.)  
The new intensity matrix $A_{k+2}$ has the same eigenvalues as the old one, but it
has the advantage that it is homogeneous of degree 1 in $\rho$ and $\chi$,
\ie, each entry is a linear combination of $\chi$ and $\rho$.
\end{rem}

\begin{rem}
  In both cases, we also obtain asymptotic joint normal distributions of the
  numbers $D_{n,k}$ and $\hD_{n,k}$ for different $k$.
\end{rem}

\begin{example}\label{Eleaves}
  The simplest case is $k=0$, when $\hD_{n,0}$ is the number of leaves in
  $\gL_n$.
In this case, \eqref{mstDpref} yields
\begin{equation}
  \mu_0
=\frac{w_1}{w_0+w_1}
=\frac{\chi+\rho}{\chi +2\rho}
=\frac{1}{\gk+1}.
\end{equation}
Moreover, the proof above yields the urn with two colours 
and intensity matrix $A_2$  given by \eqref{A2pref}.
The eigenvalues of $A_2$ are 
by direct calculation or by \refT{thm:eigenvaluesrecursive}, 
$\gl_1=\chi+\rho$ and $-\rho$.
Simple calculations show that
\begin{align}\label{Bprefleaf}&
 B=\left(
\begin{array}{cc}
 \frac{\rho+\chi}{2 \rho+\chi} & \frac{\chi (\rho+\chi)}{2 \rho+\chi} \\[10pt]
 \frac{\chi (\rho+\chi)}{2 \rho+\chi} & \frac{(\rho+\chi)(\rho^2+ \chi \rho+ \chi^2)}{2 \rho+\chi} \\
\end{array}
\right)
\end{align}
and the covariance matrix $ \Sigma $ is,
 using, for example,
\refT{thm:normality}\ref{T0c} again,
\begin{align}\label{covprefleaf}&
\Sigma=\left(
\begin{array}{cc}
 \frac{\gk}{(\gk+1)^2 (2 \gk+1)} & \frac{\gk^2}{(\gk-1) (\gk+1)^2 (2 \gk+1)} 
\\[10pt]
 \frac{\gk^2}{(\gk-1) (\gk+1)^2 (2 \gk+1)} & \frac{\gk^3}{(2 \gk+1)
   \left(\gk^2-1\right)^2}
\end{array}
\right).
\end{align}
Thus, 
the asymptotic variance of the number of leaves is
\begin{align}\label{simpleleafpref}
\gss_1=(1, 0)\Sigma (1, 0)'=\frac{\gk}{(\gk+1)^2 (2 \gk+1)},
 \end{align}
which equals the result in (\ref{gss1pref}).
(This was shown
for the random recursive tree
and the binary search tree
in \cite{Devroye1},
and for the plane oriented recursive tree 
in
\cite{MahmoudSS}.)
\end{example}

\begin{example}\label{Edegrees}
  We also consider the case when $ k=1 $. In this case the proof above yields the urn with three colours and intensity matrix $A_3$  given by 
\begin{align}\label{degpref}A_3=\left(
  \begin{array}{ccc}
 0 & \rho+\chi & 1 \\
 \rho & -\rho-\chi & 0 \\
 0 & (\rho+\chi) (\rho+2 \chi) & \chi \\
\end{array}
\right) .
\end{align}
The eigenvalues of $A_3$ are
$\gl_1=\chi+\rho$, $-\rho$ and $-(\chi+\rho)$.
Simple calculations show that
\begin{align}\label{Bdeg}
B=\left(
\begin{array}{ccc}
 \frac{\rho+\chi}{2 \rho+\chi} & -\frac{\rho}{2 (2 \rho+\chi)} & \frac{(\rho+\chi) (\rho+2 \chi)}{2 (2 \rho+\chi)} \\[10pt]
 -\frac{\rho}{2 (2 \rho+\chi)} & \frac{3 \rho}{2 (2 \rho+\chi)} & -\frac{\rho (\rho+2 \chi)}{2 (2 \rho+\chi)} \\[10pt]
 \frac{(\rho+\chi) (\rho+2 \chi)}{2 (2 \rho+\chi)} & -\frac{\rho (\rho+2 \chi)}{2 (2 \rho+\chi)} & \frac{(\rho+\chi)^2
   (\rho+2 \chi)}{2 (2 \rho+\chi)} \\
\end{array}
\right)
\end{align}
and, 
for example using
\refT{thm:normality}\ref{T0c},
\begin{align}\label{covprefdeg}&
\hskip-2mm
\Sigma=\left(
\begin{array}{ccc}
 \frac{\kappa}{(\kappa+1)^2 (2 \kappa+1)} 
& -\frac{\kappa \left(2 \kappa^2+3 \kappa+2\right)}
  {2 (\kappa+1)^2 (\kappa+2)   (2 \kappa+1)} 
& \frac{ \kappa(2-\kappa)}{2 (\kappa+1)^2 (\kappa+2) (2 \kappa+1)} 
\\[10pt]
-\frac{\kappa \left(2 \kappa^2+3 \kappa+2\right)}
{2 (\kappa+1)^2 (\kappa+2) (2 \kappa+1)} 
& \frac{\kappa \left(2  \kappa^3+25 \kappa^2+32 \kappa+12\right)}
{12 (\kappa+1)^2 (\kappa+2) (2 \kappa+1)} 
& -\frac{\kappa 
 \left(2- \kappa \right)  \left( 10\kappa^{2}+13\kappa+6 \right)} 
{12 (\kappa+1)^2 (\kappa+2) (2 \kappa+1)} 
\\[10pt]
 \frac{ \kappa(2-\kappa)}{2 (\kappa+1)^2 (\kappa+2) (2 \kappa+1)} 
& -\frac{\kappa 
 \left(2- \kappa \right)  \left( 10\kappa^{2}+13\kappa+6 \right)} 
{12 (\kappa+1)^2 (\kappa+2) (2 \kappa+1)} 
& \frac{\kappa 
 \left( 2-\kappa \right)  \left( 10\kappa^{2}+7\kappa+6 \right)}
{12 (\kappa+1)^2 (\kappa+2) (2 \kappa+1)}
\end{array}
\right).
\end{align}
In the upper left corner, we find again $\gss_1$ 
in (\ref{gss1pref}) and (\ref{simpleleafpref}).

The upper left $2\times2$ submatrix of $\gS$, giving the asymptotic
variances and covariance of the numbers of nodes of out-degrees 1 and 2,
was  found
in \cite{Devroye1} for the binary search tree ($\gk=2$),
in \cite{MahmoudSmythe} for the random recursive tree ($\gk=1$)
and in \cite{MahmoudSS} for the plane oriented recursive tree ($\gk=1/2$),
see also \cite{SJ155}.
\end{example}

%

 \newpage

\appendix
\section{The covariance matrix $\Sigma$ in Section \ref{ex2}}\label{appA}
 
\rotatebox{270}{$\Sigma=$\\
${\left(
\arraycolsep=0 pt\def\arraystretch{3.6}
\begin{array}{ccccccccc}
 \frac{3400704921}{13887473600} & \frac{54821229}{3738935200} &
   -\frac{55644023}{2991148160} & -\frac{5396373387}{194424630400} &
   -\frac{47473653}{6943736800} & -\frac{47473653}{6943736800} &
   -\frac{19950493}{7776985216} & -\frac{19950493}{7776985216} &
   -\frac{228203991}{194424630400} \\
 \frac{54821229}{3738935200} & \frac{7300603}{66040975} &
   -\frac{8044611}{325124800} & -\frac{2117030913}{97212315200} &
   -\frac{1469561}{301901600} & -\frac{1469561}{301901600} &
   -\frac{7204037}{4226622400} & -\frac{7204037}{4226622400} &
   -\frac{1334771}{1767496640} \\
 -\frac{55644023}{2991148160} & -\frac{8044611}{325124800} &
   \frac{118631347}{2113311200} & -\frac{2729459079}{194424630400} &
   -\frac{3532337}{1207606400} & -\frac{3532337}{1207606400} &
   -\frac{161089}{162562400} & -\frac{161089}{162562400} &
   -\frac{83883901}{194424630400} \\
 -\frac{5396373387}{194424630400} & -\frac{2117030913}{97212315200} &
   -\frac{2729459079}{194424630400} & \frac{617325}{17359342} &
   -\frac{7661559}{3967849600} & -\frac{7661559}{3967849600} &
   -\frac{123349341}{194424630400} & -\frac{123349341}{194424630400} &
   -\frac{3764589}{13887473600} \\
 -\frac{47473653}{6943736800} & -\frac{1469561}{301901600} &
   -\frac{3532337}{1207606400} & -\frac{7661559}{3967849600} &
   \frac{63201}{8625760} & -\frac{3151}{8625760} & -\frac{4847}{41641600} &
   -\frac{4847}{41641600} & -\frac{193079}{3967849600} \\
 -\frac{47473653}{6943736800} & -\frac{1469561}{301901600} &
   -\frac{3532337}{1207606400} & -\frac{7661559}{3967849600} &
   -\frac{3151}{8625760} & \frac{63201}{8625760} & -\frac{4847}{41641600} &
   -\frac{4847}{41641600} & -\frac{193079}{3967849600} \\
 -\frac{19950493}{7776985216} & -\frac{7204037}{4226622400} &
   -\frac{161089}{162562400} & -\frac{123349341}{194424630400} &
   -\frac{4847}{41641600} & -\frac{4847}{41641600} & \frac{157523}{72872800} &
   -\frac{2637}{72872800} & -\frac{2884319}{194424630400} \\
 -\frac{19950493}{7776985216} & -\frac{7204037}{4226622400} &
   -\frac{161089}{162562400} & -\frac{123349341}{194424630400} &
   -\frac{4847}{41641600} & -\frac{4847}{41641600} & -\frac{2637}{72872800} &
   \frac{157523}{72872800} & -\frac{2884319}{194424630400} \\
 -\frac{228203991}{194424630400} & -\frac{1334771}{1767496640} &
   -\frac{83883901}{194424630400} & -\frac{3764589}{13887473600} &
   -\frac{193079}{3967849600} & -\frac{193079}{3967849600} &
   -\frac{2884319}{194424630400} & -\frac{2884319}{194424630400} &
   \frac{5681341}{6943736800} \\
\end{array}
\right)}$}

\section{The covariances in the matrix $\Sigma$ in Section \ref{ex3}}\label{appB}

\begin{align*}
\Sigma_{1,1}&=\frac{\gk \left(-14 \gk^5-73 \gk^4+131 \gk^3+1438 \gk^2+3018 \gk+2070\right)}{(\gk+1) (\gk+2)^2 (\gk+3)^2 (2 \gk+1) (2 \gk+3) (2 \gk+5)};
   \nonumber\\\Sigma_{1,2}&=\frac{3 \gk \left(2 \gk^6+37 \gk^5+124 \gk^4-55 \gk^3-900 \gk^2-1488 \gk-720\right)}{4 (\gk+1)^2 (\gk+2)^2 (\gk+3)^2 (2 \gk+1) (2 \gk+3) (2 \gk+5)};
     \nonumber\\\Sigma_{1,3}&=\frac{\gk^2 \left(10 \gk^5-47 \gk^4-664 \gk^3-2083 \gk^2-2592 \gk-1080\right)}{4 (\gk+1)^2 (\gk+2)^2 (\gk+3)^2 (2 \gk+1) (2 \gk+3) (2 \gk+5)};
      \nonumber\\\Sigma_{1,4}&=\frac{\gk \left(10 \gk^5-47 \gk^4-664 \gk^3-2083 \gk^2-2592 \gk-1080\right)}{4 (\gk+1)^2 (\gk+2)^2 (\gk+3)^2 (2 \gk+1) (2 \gk+3) (2 \gk+5)};
       \nonumber\\\Sigma_{1,5}&=-\frac{\gk \left(20 \gk^6+104 \gk^5-69 \gk^4-1088 \gk^3-1307 \gk^2+1224 \gk+2160\right)}{2 (\gk-1) (\gk+1) (\gk+2)^2 (\gk+3)^2 (2 \gk+1) (2 \gk+3) (2 \gk+5)};
        \nonumber\\\Sigma_{2,2}&=\frac{3 \gk \left(17 \gk^6+145 \gk^5+507 \gk^4+929 \gk^3+976 \gk^2+612 \gk+180\right)}{2 (\gk+1)^2 (\gk+2)^2 (\gk+3)^2 (2 \gk+1) (2 \gk+3) (2 \gk+5)};
         \nonumber\\\Sigma_{2,3}&=-\frac{\gk^3 \left(80 \gk^4+579 \gk^3+1558 \gk^2+1803 \gk+720\right)}{4 (\gk+1)^2 (\gk+2)^2 (\gk+3)^2 (2 \gk+1) (2 \gk+3) (2 \gk+5)};
          \nonumber\\\Sigma_{2,4}&=-\frac{\gk^2 \left(80 \gk^4+579 \gk^3+1558 \gk^2+1803 \gk+720\right)}{4 (\gk+1)^2 (\gk+2)^2 (\gk+3)^2 (2 \gk+1) (2 \gk+3) (2 \gk+5)};
           \nonumber\\\Sigma_{2,5}&=\frac{\gk \left(28 \gk^7+264 \gk^6+829 \gk^5+816 \gk^4-515 \gk^3-702 \gk^2+1152 \gk+1080\right)}{4 (\gk-1) (\gk+1)^2 (\gk+2)^2 (\gk+3)^2 (2 \gk+1) (2 \gk+3) (2 \gk+5)};
            \nonumber\\\Sigma_{3,3}&=\frac{\gk^2 \left(49 \gk^4+276 \gk^3+539 \gk^2+396 \gk+90\right)}{2 (\gk+1)^2 (\gk+2) (\gk+3)^2 (2 \gk+1) (2 \gk+3) (2 \gk+5)};
    \nonumber\\\Sigma_{3,4}&=-\frac{\gk^3 \left(16 \gk^3+87 \gk^2+152 \gk+75\right)}{2 (\gk+1)^2 (\gk+2) (\gk+3)^2 (2 \gk+1) (2 \gk+3) (2 \gk+5)};
     \nonumber\\\Sigma_{3,5}&=-\frac{\gk^2 \left(4 \gk^6+64 \gk^5+251 \gk^4+194 \gk^3-729 \gk^2-1488 \gk-720\right)}{4 (\gk-1) (\gk+1)^2 (\gk+2)^2 (\gk+3)^2 (2 \gk+1) (2 \gk+3) (2 \gk+5)};
      \nonumber\\\Sigma_{4,4}&=\frac{\gk \left(16 \gk^5+120 \gk^4+341 \gk^3+462 \gk^2+321 \gk+90\right)}{2 (\gk+1)^2 (\gk+2) (\gk+3)^2 (2 \gk+1) (2 \gk+3) (2 \gk+5)};
       \nonumber\\\Sigma_{4,5}&=-\frac{\gk \left(4 \gk^6+64 \gk^5+251 \gk^4+194 \gk^3-729 \gk^2-1488 \gk-720\right)}{4 (\gk-1) (\gk+1)^2 (\gk+2)^2 (\gk+3)^2 (2 \gk+1) (2 \gk+3) (2 \gk+5)};
        \nonumber\\\Sigma_{5,5}&=\frac{\gk \left(-4 \gk^7-4 \gk^6+147 \gk^5+574 \gk^4+610 \gk^3-51 \gk^2+132 \gk+630\right)}{(\gk-1)^2 (\gk+1) (\gk+2)^2 (\gk+3)^2 (2 \gk+1) (2 \gk+3) (2 \gk+5)}
.
\end{align*}

\end{document}